\documentclass[11pt,reqno]{amsart}

\usepackage[margin=1in]{geometry}
\usepackage{xcolor}
\usepackage{hyperref}

\usepackage{amsfonts,amssymb}
\usepackage[backrefs]{amsrefs}
\usepackage{amsmath, verbatim,amsxtra}
\usepackage{amstext}
\usepackage{amsthm}
\usepackage{enumitem}
\usepackage{mathrsfs}
\usepackage{mathtools}
\usepackage{bbm}

\theoremstyle{plain}
\newtheorem{thm}{Theorem}
\newtheorem{prop}[thm]{Proposition}
\newtheorem{cor}[thm]{Corollary}
\newtheorem{lem}[thm]{Lemma}
\newtheorem{clm}[thm]{Claim}

\newtheorem*{conjecture}{Conjecture}
\theoremstyle{remark}
\newtheorem*{rmk}{Remark}
\newtheorem*{rmks}{Remarks}
\theoremstyle{definition}

\newcommand{\C}{\mathbb C}

\DeclareMathOperator{\sinc}{sinc}
\DeclareMathOperator{\Leb}{Leb}

\DeclareMathOperator{\var}{Var}

\newcommand{\ind}{{\mathbbm{1}}}

\newcommand{\cN}{{\mathcal N}}
\newcommand{\cF}{{\mathcal F}}
\newcommand{\R}{\mathbb R}

\newcommand{\E}{\mathbb E}
\newcommand{\N}{\mathbb N}
\newcommand{\Z}{\mathbb Z}
\newcommand{\Pro}{\mathbb P}
\newcommand{\lm}{\lambda}
\newcommand{\si}{\sigma}

\newcommand{\ep}{\varepsilon}
\newcommand{\p}{\varphi}
\newcommand{\al}{\alpha}

\newcommand{\M}{\mathcal{M}(\R)}

\title[Variance of zeroes of a stationary Gaussian process]{An asymptotic formula for the variance of the number of zeroes of a stationary Gaussian process}

\author{Eran Assaf}
\author{Jeremiah Buckley}
\author{Naomi Feldheim}

\address{E. Assaf, Dartmouth College, New Hampshire, USA}
\email{eran.assaf@dartmouth.edu}
\address{J. Buckley, King's College, London, United Kingdom}
\email{jeremiah.buckley@kcl.ac.uk}
\address{N. Feldheim, Bar-Ilan University, Ramat-Gan, Israel}
\email{naomi.feldhim@biu.ac.il}

\date{}

\begin{document}
\begin{abstract}
We study the variance of the number of zeroes of a stationary Gaussian process on a long interval. We give a simple asymptotic description under mild mixing conditions. This allows us to characterise minimal and maximal growth. We show that a small (symmetrised) atom in the spectral measure at a special frequency does not affect the asymptotic growth of the variance, while an atom at any other frequency results in maximal growth. Our results allow us to analyse a large number of interesting examples. 

\vspace{5pt}
\end{abstract}
\maketitle

\section{Introduction}

Zeroes of Gaussian processes, and in particular stationary Gaussian processes (SGPs), have been widely studied, with diverse applications in physics and signal processing.
The expected number of zeroes may be computed by the celebrated Kac-Rice formula. Estimating the fluctuations, however, proved to be a much more difficult task. For a comprehensive historical account see \cite{Kratz}.

The aim of this paper is to give a simple expression which describes the growth of the variance of the number of zeroes in the interval $[0,T]$, as $T\to\infty$. Following the ideas of Slud~\cite{Slud94}, it is easy to give a lower bound for this quantity. Our main contribution is a matching upper bound, which holds under a very mild hypothesis. In particular we give a sharp asymptotic expression for the variance for any process with decaying correlations, no matter how slow the decay. 

An intriguing feature of our results is the emergence of a `special frequency': adding an atom to the spectral measure at this frequency does not change the order of growth of the fluctuations.

\subsection{Results}
Let $f:\R\to\R$ be a stationary Gaussian process (SGP) with continuous \emph{covariance kernel} \[r(t) = \E[f(0) f(t)].\]
Denote by $\rho$ the \emph{spectral measure} of the process, that is, the unique finite, symmetric measure on $\R$ such that
\[
r(t) = \cF[\rho](t)= \int_\R e^{-i \lm t} d\rho(\lm).
\]
We normalise the process so that $r(0) =\rho(\R)= 1$. It is well-known (see, e.g., \cite{CL}*{Section 7.6}) that the distribution of $f$ is determined by $\rho$, and further that any such $\rho$ is the spectral measure of some SGP.

We study the number of zeroes of $f$ in a long `time' interval $[0,T]$, which we denote
\[
N(\rho;T) = N(T) =\# \{ t\in [0,T]: \: f(t)=0\}.
\]
The expectation of $N(T)$ is given by the Kac-Rice formula (see~\cites{Ylv,Ito})
\begin{equation}\label{eq: KR}
\E [N(T)] = \frac{\si}{\pi} T,
\end{equation}
where
\begin{equation*}
 \si^2 = -r''(0) = \int_\R \lm^2 d\rho(\lm).
\end{equation*}
Throughout we assume that $N(T)$ has finite variance, which turns out to be equivalent to the \emph{Geman condition}~\cite{Gem}
\begin{equation}\label{eq: Geman}
\int_0^{\ep} \frac{r''(t) - r''(0) }{t} dt < \infty\quad\text{for some}\quad\ep>0.
\end{equation}

\smallskip
An SGP $f$ is \emph{degenerate} if its spectral measure consists of a single symmetrised atom $\rho = \delta^*_\alpha = \frac 1 2 (\delta_{\alpha} + \delta_{-\alpha})$, or equivalently if the covariance is $r(t) = \cos(\al t)$. 
In this case the zero set is a random shift of the lattice $\frac {\pi}{\al}\Z$, and the variance $\var[N(T)]$ is bounded.

\smallskip
We formulate our results in terms of the function
\begin{equation} \label{eq: phi}
\p(t) = \max \left\{ |r(t)| + \frac{|r'(t)|}{\sigma}, \, \frac{|r'(t)|}{\sigma} + \frac{|r''(t)|}{\sigma^2} \right\}.
\end{equation}
The notation $A(T) \asymp B(T)$ denotes that there exist $C_1, C_2>0$ such that
$C_1 \le \frac{A(T)}{B(T)}\le C_2$ for all $T>0$, while $A(T)\sim B(T)$ denotes that
$\lim_{T\to\infty} \frac{A(T)}{B(T)} = 1$. Our main result is the following.

\begin{thm}\label{thm: UB}
For any SGP satisfying
\begin{equation}\label{eq: cond}
\limsup_{|t|\to\infty} \p(t) <1,
\end{equation}
we have
\begin{equation}\label{eq: var main}
\var[N(T) ] \asymp  T\int_{0}^T \left(1-\frac{t}{T}\right) \left(r(t)+\frac{r''(t)}{\si^2}\right)^2 \ dt
\end{equation}
where the implicit constants depend on $\rho$.
\end{thm}

The condition \eqref{eq: cond} may be viewed as a very mild mixing condition, which in particular holds whenever the spectral measure is absolutely continuous. In fact, the condition $r(t)\overset{|t|\to\infty}{\longrightarrow} 0$ implies that $\varphi(t)\overset{|t|\to\infty}{\longrightarrow}0$. 

Under some additional conditions, we are able to compute the leading constant in~\eqref{eq: var main}.
\begin{thm} \label{thm : asymp constant}
Suppose that $r+\frac{r''}{\si^2} \not\in \mathcal{L}^2(\R)$, and
$\lim_{|t|\to\infty} \p(t) = 0$.
Then
\[
\var[N(T) ] \sim \frac{\sigma^2}{\pi^2} T\int_{0}^T \left(1-\frac{t}{T}\right) \left(r(t)+\frac{r''(t)}{\si^2}\right)^2 dt .
\]
\end{thm}

In fact, the lower bound in this theorem holds for any process. We state this separately.

\begin{prop} \label{prop: LB}
For any SGP,
\[
\var[N(T)] \ge  \frac{\sigma^2}{\pi^2} T\int_{0}^T \left(1-\frac{t}{T}\right) \left(r(t)+\frac{r''(t)}{\si^2}\right)^2 \ dt.
\]
In particular, for any non-degenerate SGP there exists a constant $C=C(\rho)>0$ such that
\begin{equation}\label{eq: linlowerbound}
  \var[N(T)] \ge C T, \quad \forall T>0.
\end{equation}
\end{prop}

The variance for non-degenerate SGPs is therefore always at least linear in $T$. By stationarity, it is at most quadratic. Our next result characterises these two extremes.

\begin{thm}\label{thm: extremal} 
\leavevmode
\begin{enumerate}[label=(\alph*)]
\item\label{part: lin}
Suppose that condition \eqref{eq: cond} holds. Then
\[ \var[N(T)] \asymp T\iff r+ \frac{r''}{\si^2} \in \mathcal{L}^2(\R).\]
\item\label{part: quad} $\var[N(T)]  \asymp T^2$ if and only if $\rho$ contains an atom at a point different from $\si$.
\end{enumerate}
\end{thm}


The emergence of a special frequency $\si$ in Theorem~\ref{thm: extremal}~\ref{part: quad} is new, and intriguing. One naturally asks what the effect of an atom at this frequency is. Notice that modifying a measure by adding an atom at frequency $\si$ does not change $\E [N(T)]$. The following result follows from Theorem~\ref{thm: UB}, and shows that the asymptotic growth of $\var N(T)$ remains unchanged as well --- at least under some mild assumptions.

\begin{cor} \label{cor: special atom}
Suppose that \eqref{eq: cond} holds for the spectral measure $\rho$. Define\footnote{Notice that $\E[N(\rho_\theta; T)]$ is independent of $\theta$.} $\rho_\theta=(1-\theta) \rho +  \theta \delta^*_\si$ for $0<\theta<1$. There exists $\theta_0>0$ such that
\[
\var[N(\rho; T)] \asymp \var[ N(\rho_\theta; T)]
\]
for any $\theta<\theta_0$ (and the implicit constants may depend on $\theta$). Moreover, $\theta_0$ depends only on $\limsup_{|t| \rightarrow \infty} \varphi(t)$.
\end{cor}

\subsection{Discussion}

It is evident from our results that the decay of $r(t)+\frac{r''(t)}{\si^2}$ determines the variance of $N(T)$. 
Observe that $r + \frac{r''}{\sigma^2}=\cF[\mu]$ where the signed measure $\mu$ is defined by $d\mu(\lm)=\left(1 - \frac{\lambda^2}{\sigma^2} \right) d \rho(\lambda)$; this is crucial to some of our proofs. In fact, it follows from Parseval's identity that
    \begin{equation}\label{eq: parseval for key int}
    \int_{0}^T \left(1-\frac{t}{T}\right) \left(r(t)+\frac{r''(t)}{\si^2}\right)^2 \ dt = \pi  \int (\mathcal{S}_T * \mu)\ d\mu
    \end{equation}
where $\mathcal{S}_T(\lambda) =  \frac{T}{2\pi} \  \sinc^2\left(\frac{T \lm}{2}\right).$ For details, see Section~\ref{sec: quad}. 
Cancellation between the terms of $r(t) + \frac{r''(t)}{\si^2}$ plays an important r\^{o}le in determining the variance. This is further emphasised in Section~\ref{sec: Cancellation}. One consequence of this cancellation is the emergence of the special atom (in the sense of Theorem~\ref{thm: extremal}~\ref{part: quad} and Corollary~\ref{cor: special atom}). This phenomenon is explained, in part, by the fact that the measure $\mu$ does not `see' $\si$.

For crossings of non-zero levels, the presence of an atom at any frequency leads to quadratic variance. The existence of a special atom at a distinguished frequency is therefore unique to the zero level. Furthermore, this phenomenon is purely real. No such frequency exists for complex zeroes, see~\cite{Nvar}.
    
We remark that, following Arcones \cite{A}, many previous results were stated in terms of the function
        \begin{equation*}
          \psi(t) = \max\left\{|r(t)|, \frac{|r'(t)|}{\si}, \frac{|r''(t)|}{\si^2} \right\}
        \end{equation*}
        rather than the function $\p$ that we introduced in \eqref{eq: phi}. For purposes of comparison, our assumption \eqref{eq: cond} is implied by the stronger assumption  $\limsup_{|t|\to\infty} \psi(t) <\frac 1 {2}$.

While the condition \eqref{eq: cond} is a very mild mixing condition, there are some processes with singular spectral measure for which it does not hold. We believe that the result holds in greater generality.

\begin{conjecture}[Weak form]
The estimate \eqref{eq: var main} holds for any non-degenerate SGP satisfying
\begin{equation}\label{eq: old phi}
  \limsup_{|t|\to\infty} \max \left\{ r(t)^2 + \frac{r'(t)^2}{\sigma^2}, \, \frac{r''(t)^2}{\sigma^4} + \frac{r'(t)^2}{\sigma^2} \right\}<1.
\end{equation}
\end{conjecture}
\begin{conjecture}[Strong form]
The estimate \eqref{eq: var main} holds for any non-degenerate SGP.
\end{conjecture}
\noindent Even the weak form of the conjecture would allow us to prove stronger results, e.g., improve Theorem~\ref{thm: extremal}~\ref{part: lin} to completely characterise linear variance and prove that Corollary~\ref{cor: special atom} holds for any $\theta \in [0,1)$. We provide further evidence for the conjectures in Section~\ref{section: conjectures}.

\subsection{Background and motivation}

The origins for the Kac-Rice method for computing the expected number of zeroes lie in the independent work of Kac~\cites{Kac43, Kac48} and of Rice~\cites{Rice44, Rice45}. Applying this method to SGPs yields the formula~\eqref{eq: KR}, even when both sides are infinite, as was done by Ylvisaker~\cite{Ylv} and Ito \cite{Ito}.  
Sufficiency of the Geman condition~\eqref{eq: Geman} for finite variance was proved by
Cram\'{e}r and Leadbetter \cite{CL}*{Equation 10.6.2 or 10.7.5}, while necessity was established by Geman~\cite{Gem}.
Qualls~\cite{QuPhD}*{Lemma 1.3.4} showed that 
the Geman condition is equivalent to the spectral condition
$
  \int_\R \log(1+|\lm|)\lm^2 d\rho(\lm)<\infty
$ (see also \cite{Boas}*{Theorem 3}).

An exact but somewhat inaccessible formula for the variance was rigorously derived\footnote{This formula was based on the ideas of Kac-Rice, and indeed such a formula was known to physicists \cite{SSZW} and had been proved mathematically assuming the existence of $r^{(vi)}(0)$ (see the footnote on \cite{VR}*{Page 188}).} by Cram\'{e}r and Leadbetter \cite{CL}*{Sections 10.6-7}. Little progress in understanding the asymptotic growth of the variance was made until Slud~\cites{Slud91,Slud94} introduced Multiple Wiener Integral techniques some decades later --- these were in turn refined and extended by Kratz and L\'eon~\cites{KL06, KL10}, using Wiener chaos expansions.


The formulas mentioned above were used to prove various properties of the zeroes, such as sufficient conditions for linearity of the variance and for a central limit theorem (see, e.g., \cites{Cuz,Mal}). However, extracting the asymptotic growth of the variance under reasonably general conditions has proved fruitless. For example, the only attempt at a systematic study of super-linear growth of the variance that we are aware of is a special family of examples due to Slud \cite{Slud94}*{Theorem 3.2}. In Section~\ref{sec: inter exa} we use Theorem~\ref{thm : asymp constant} to improve Slud's result.
More generally, one can now analyse a large number of examples, as we do in Section~\ref{sec: examples}, due to the simplicity of the quantity $r + r'' / {\si}^2$ which appears throughout our results.

The case of linear variance was historically of interest. Previously, the only condition for asymptotically linear variance (that we are aware of)
was $r,r''\in \mathcal{L}^2(\R)$, which follows from combining the results of Cuzick~\cite{Cuz} and Slud~\cite{Slud91}. We show in
Section~\ref{sec: Cancellation} that the condition $r+ \frac{r''}{\si^2} \in \mathcal{L}^2(\R)$ is strictly weaker, therefore Theorem~\ref{thm: extremal}~\ref{part: lin} improves upon their result. It also follows from their work that $r, r'' \in \mathcal{L}^2(\R)$ implies that $\frac{1}T\var[N(T)]$ converges as $T\to\infty$. Ancona and Letendre \cite{AL}*{Proposition 1.11} give an exact expression for this limit (see also \cite{Dal}*{Proposition 3.1}), although their main focus is on the growth of the central moments of linear statistics (which generalise the zero count). The lower bound \eqref{eq: linlowerbound} also appears in the work of Lachi{\`e}ze-Rey \cite{LR}, who studies rigidity and predictability of the zero set. 

We finally mention that our work has parallels in different but related models. In the setting of complex zeroes of a random Gaussian analytic $f:\C\to\C$ an asymptotic formula for the variance, an $L^2$-condition that guarantees linearity, and a characterisation of maximal (i.e., quadratic) growth were given in \cite{Nvar}. Analogous results were then proved for the winding number of a Gaussian stationary $f:\R\to\C$ in \cite{BF}.

\subsection{Outline of our methods}\label{subsec: outline}
Let us briefly outline our method. We write
\begin{equation*}
  N(T)=\sum_{q=0}^\infty \pi_{q}(N(T))
\end{equation*}
where $\pi_q$ denotes the projection onto the $q$'th Wiener chaos. Explicit expressions for this decomposition are well known, it turns out that only the even chaoses contribute, and so we have
\begin{equation*}
  \var[N(T)]=\sum_{q=1}^\infty \E[\pi_{2q}(N(T))^2].
\end{equation*}
The diagram formula allows us to compute (see Lemma~\ref{lem: ENq^2})
\begin{equation*}
  \E[\pi_{2q}(N(T))^2]=\int_{-T}^T (T-|t|)\widetilde{P}_q(t) dt
\end{equation*}
where $\widetilde{P}_q$ is a polynomial expression that involves $r,r'$ and $r''$. Our lower bound comes from explicitly evaluating the term with $q=1$. For the upper bound we establish that $\left(r+\frac{r''}{\si^2}\right)^2$ divides the polynomial\footnote{Strictly speaking we first add a small computable quantity, which leads to the difference between $P_q$ and $\widetilde{P}_q$ in Section~\ref{sec: chaos expan}.} $\widetilde{P}_q$ exactly, see Proposition~\ref{prop: divisibility}. This yields
\begin{equation*}
  \E[\pi_{2q}(N(T))^2]\le C_q \int_{-T}^T (T-|t|) \left(r(t)+\frac{r''(t)}{\si^2}\right)^2 dt
\end{equation*}
for some $C_q$. The remainder of our proof of the upper bound involves showing that this sequence $C_q$ is summable under the given hypothesis.

\subsection*{Acknowledgements}
Mikhail Sodin first suggested the project to us.
We had a number of interesting and fruitful discussions with Yan Fyodorov, Marie Kratz and Igor Wigman on various topics related to this article. Eugene Shargorodsky contributed the main idea in the proof of Claim~\ref{clm: eugene}. We are thankful to Ohad Feldheim for suggesting improvements to the presentation of this paper. The research of E.A. is supported by a Simons Collaboration Grant (550029, to John Voight), his research was partly conducted while hosted in King's College London and supported by EPSRC grant EP/L025302/1. The research of J.B. is supported in part by EPSRC New Investigator Award EP/V002449/1.
The research of N.F. is partially supported by Israel Science Foundation grant 1327/19.

\section{Examples}\label{sec: examples}
In this section we give a number of examples that expand on or illustrate our results.

\subsection{Purely atomic measure}
If $\rho$ consists of a single symmetrised atom $\rho = \delta_{\\al}^*$, then
$f(t)$
is a degenerate process and thus $\var [N(T)]$ is bounded. However, a superposition of such processes results in a random almost periodic function, with non-trivial behavior.
Specifically let $\alpha_i \in\R $ and $w_i>0$ satisfy $\sum_{i} w_i=1$ and $\sum_{i} w_i \alpha_i^2 \log(1+|\alpha_i|)<+\infty$ \footnote{There might only be finitely many $\alpha_i$ in which case this second condition is redundant.}. We consider the measure
\begin{equation*}
  \rho =  \sum_{i} w_i \delta_{\alpha_i}^*
\end{equation*}
Then the covariance function is
$r(t) = \sum_{i} w_i \cos(\alpha_i t) $ which yields $r''(t) = - \sum_{i} w_i \alpha_i^2 \cos(\alpha_i t) $ and $\sigma^2 = -r''(0) =  \sum_{i} w_i \alpha_i^2$. By Theorem~\ref{thm: extremal}~\ref{part: quad} we have quadratic growth $\var [N(T)] \asymp T^2$. Using Proposition ~\ref{prop: LB}, we may give a concrete lower bound, that is,
\[
\liminf_{T\to\infty}\frac{\var [N(T)]}{T^2} \ge \frac{\si^2}{4 \pi^2} \sum_{i=1}^n w_i^2 \left(1-\frac{\al_i^2}{\si^2}\right)^2;
\]
we omit the details.

\subsection{Cuzick--Slud covariance functions}\label{sec: exa square int}
We have already discussed in detail the case $r, r'' \in \mathcal{L}^2(\R)$; in this case the limit
\begin{equation*}
  \lim_{T\to\infty} \frac{\var[N(T)]}T
\end{equation*}
exists and we further have a CLT for $N(T)$. We mention that a number of classic kernels satisfy this hypothesis, for example, the Paley-Wiener kernel $\sinc t = \frac{\sin t}{t}$ or the Gaussian kernel (sometimes referred to as the Bargman-Fock kernel) $e^{-t^2}$.

\subsection{Exponential kernel and approximations}
Consider the Ornstein-Uhlenbeck (OU) process, defined by the covariance function $r(t) = e^{-|t|}$. This process has attracted considerable attention since it arises as a time-space change of Brownian motion. Since the covariance is not differentiable at the origin, none of our results may be directly applied. However, one may approximate the OU process by differentiable processes.

One way to do so is by taking
$ r_a(t) = e^{a-\sqrt{a^2+t^2}} $, with $a \downarrow 0$. In this case $\sigma_a^2 = \frac{1}{a}$ and
\begin{equation*}
  r_a(t) + \frac{r_a''(t)}{\sigma_a^2}=\left(1-\frac{a^3}{(a^2+t^2)^{3/2}}+\frac{at^2}{a^2+t^2}\right)r_a(t)\ge e^{-t}
\end{equation*}
where the inequality holds for $t\ge a^{2/3}$. We deduce from Proposition~\ref{prop: LB} that for $T\ge a^{2/3}$ we have
\begin{equation*}
  \var[N(T)]\ge\frac T{\pi^2 a}\int_{a^{2/3}}^T \left(1-\frac{t}{T}\right) e^{-2t} \ dt.
\end{equation*}
As $a \downarrow 0$, we see that the variance is unbounded, and this holds even on certain short intervals that
are not `too short', i.e, such that $T \gg \sqrt a$.

Another approximation may be derived using the spectral measure. The OU process has spectral density
$\frac{1}{\pi (1+ \lambda^2)}$. Thus one may consider the spectral density
$\frac{M}{\pi(M-1)} \left(\frac{1}{\lambda^2+1} - \frac{1}{\lambda^2+M^2} \right)$, with $M~\rightarrow~\infty$.
The corresponding covariance kernel is $r_M(t) = \frac{Me^{-|t|} - e^{-M|t|}}{M-1}$, with $\sigma_M^2 = M$.
In this case we have
$$
r_M(t) + \frac{r_M''(t)}{\sigma_M^2} = \frac{M+1}{M-1} \left( e^{-t}-e^{-Mt} \right)\ge e^{-t}-e^{-Mt}
$$
for $t>0$, so that applying Proposition~\ref{prop: LB} we obtain
\begin{equation*}
\var [N(T)] \ge \frac{MT}{\pi^2} \int_{0}^{T} \left( 1 - \frac{t}{T} \right) \left( e^{-2t}-2e^{-(M+1)t}+e^{-2Mt} \right) \ dt.
\end{equation*}
Once more we see that the variance is unbounded, even on short intervals that satisfy $T\gg\frac1{\sqrt M}$.

\subsection{Bessel function}
Let $J_{\alpha}(t)$ be the $\alpha$'th Bessel function of the first kind.
If $r(t) = J_0(t)$ then $J_0'(t) = -J_1(t)$ and
$ J_0''(t)=\frac{J_2(t) - J_0(t)}{2}$, so that $\sigma^2 = \frac{1}{2}$ and
$$
J_0(t) + \frac{J_0''(t)}{\sigma^2} = J_2(t) \notin \mathcal{L}^2(\R).
$$
Moreover, in this example $r(t) \rightarrow 0$, and so Theorem~\ref{thm : asymp constant} applies and we have
\begin{equation*}
\var [N(T)]  \sim \frac{\sigma^2}{\pi^2} T \int_{0}^{T} \left(1 - \frac{t}{T} \right) J_2(t)^2 \ dt \sim \frac{1}{2\pi^3} T \ln T.
\end{equation*}

\subsection{Intermediate growth}\label{sec: inter exa}
Using Theorem~\ref{thm : asymp constant} it is possible to construct more examples where $ T \ll \var [N(T)] \ll T^2 $. For instance, consider
\begin{equation*}
  r(t) = \frac{1}{(1+t^2)^b},
\end{equation*}
with $ 0 < b < \frac{1}{4} $. One way to see that this is a legitimate covariance is to compute the spectral measure
$d\rho_b(\lambda)=\frac{1}{\Gamma(b)\sqrt\pi} 2^{\frac12-b} |\lambda|^{b-\frac{1}{2}} K_{\frac{1}{2}-b} (|\lambda|)$, where $K_{\alpha}$ is the modified Bessel function of the second kind of order $\alpha$. Alternatively\footnote{This argument appears in \cite{Slud94}} note that
\begin{equation*}
  r(t) = \E[e^{-i(X-\tilde{X})}]
\end{equation*}
where $X$ and $\tilde{X}$ are independent $\Gamma(b,1)$ random variables. It is easy to check that the hypotheses of Theorem~\ref{thm : asymp constant} are satisfied and so
\begin{equation*}
  \var [N(T)]\sim\frac{2b}{\pi^2 (1-4b)(2-4b)} T^{2-4b}.
\end{equation*}

More generally let $L$ be a function that `varies slowly at infinity', that is, for every $x>0$ we have
\begin{equation*}
  \frac{L(tx)}{L(t)}\to1
\end{equation*}
as $t\to\infty$ and suppose that we have
\begin{equation*}
  r(t) = \frac{L(|t|)}{(1+t^2)^{b}}.
\end{equation*}
Suppose further that there exist $C>0$ and $\delta<b$ such that
\begin{equation}\label{eq: r',r'' growth Slud}
  \frac{|r'(t)|}{|r(t)|},\frac{|r''(t)|}{|r(t)|}\le C (1+t^2)^\delta.
\end{equation}
Then Slud \cite{Slud94}*{Theorem 3.2} showed that
\begin{equation*}
  \var [N(T)]\sim\frac{\sigma^2}{\pi^2 (1-4b)(2-4b)}L(T)^2 T^{2-4b}
\end{equation*}
and moreover that $N(T)$ satisfies a non-CLT.

Theorem~\ref{thm : asymp constant} allows us to prove Slud's result, without imposing the hypothesis \eqref{eq: r',r'' growth Slud}. Indeed, Slud uses this hypothesis to show that the higher order chaoses are negligible, that is, that\footnote{$N_1(T)$ will be defined in the next section.}
\begin{equation*}
  \var[N(T)]\sim \frac{\si^2}{\pi^2} \frac{\E[N_1(T)^2]}4.
\end{equation*}
But this is precisely the conclusion of Theorem~\ref{thm : asymp constant}, which applies here since $r(t)\to0$ as $t\to\infty$. To compute the asymptotic growth of the variance we write, as in \eqref{eq: parseval for key int},
\begin{equation*}
  \frac1{L(T)^2 T^{2-4b}}\int_{-T}^T (T-|t|) \left(r(t)+\frac{r''(t)}{\si^2}\right)^2 dt =\frac{2\pi T}{L(T)^2 T^{2-4b}}  \int (\mathcal{S}_T * \mu)\ d\mu
\end{equation*}
where $d\mu(\lm)=\left(1-\frac{\lm^2}{\si^2}\right)d\rho(\lm)$. We define $d\rho_T(\lm)=\frac {T^{2b}}{L(T)}d\rho\left(\frac\lm T\right)$ and $d\mu_T(\lm)=\left(1-\frac{\lm^2}{T^2\si^2}\right)d\rho_T(\lm)$. Changing variables in the previous equation we see that
\begin{equation*}
  \frac1{L(T)^2 T^{2-4b}}\int_{-T}^T (T-|t|) \left(r(t)+\frac{r''(t)}{\si^2}\right)^2 dt =2\pi \int (\mathcal{S}_1 * \mu_T)\ d\mu_T.
\end{equation*}
Now \cite{DM}*{Proposition 1} implies that $\rho_T$ converges locally weakly to the measure $\rho_0$ with density $\frac{|\lm|^{2b-1}}{2\Gamma(2b)\cos(b\pi)}$. Since $1-\frac{\lm^2}{T^2\si^2}\to1$ uniformly on bounded sets we have $\mu_T\to\rho_0$ locally weakly also. Further, by \cite{DM}*{Equation 1.10}, we have $\hat{\rho}_0(t)=|t|^{-2b}$. This yields
\begin{align*}
  \frac1{L(T)^2 T^{2-4b}}\int_{-T}^T (T-|t|)\left(r(t)+\frac{r''(t)}{\si^2}\right)^2 &\sim 2\pi \int (\mathcal{S}_1 * \rho_0)\ d\rho_0\\
  &=\int_{-1}^1 (1-|t|)|\hat{\rho}_0(t)|^2 dt\\
  &=2\int_{0}^1 \frac{1-t}{t^{4b}}dt=\frac2{(1-4b)(2-4b)}.
\end{align*}
\subsection{Singular continuous spectral measure}\label{subsec: sing cnts exa}

If the spectral measure is absolutely continuous with respect to Lebesgue measure, then \eqref{eq: cond} holds by the Riemann-Lebesgue lemma and we may apply Theorem~\ref{thm: UB}. If there are atoms in the spectral measure then Theorem~\ref{thm: extremal}~\ref{part: quad} applies. Singular continuous measures fall between these two stools. Here we give a family of examples that show that there is no simple characterisation for this class of processes.

Fix $\alpha_1 \ge \alpha_2 \ge \ldots >0$ such that $\sum_{n=1}^{\infty} \alpha_n^2 < \infty$, and let
$r_{\alpha}(t) = \prod_{n=1}^{\infty} \cos(\alpha_n t)$. In this case the spectral measure $\rho_{\alpha}$ is the infinite convolution of the atomic measures $\delta_{\alpha_n}^*$ and is usually referred to as a symmetric Bernoulli convolution. By the Jessen–Wintner theorem \cite{JW}*{Theorem 11} it is of pure type, being either absolutely continuous or singular continuous. (In particular it contains \emph{no atoms}.) Moreover, the support of $\rho_\alpha$ (the `spectrum') is a perfect set and is either compact or all of $\R$, according as $\sum_{n=1}^{\infty} \alpha_n$ converges or diverges.

In the former case we write $R_n=\sum_{i=n+1}^{\infty} \alpha_i$ and the spectrum is a subset of $[-R_0,R_0]$. There are two special cases that are particularly tractable:
\begin{enumerate}[label=\arabic*.]
  \item $\alpha_n>R_n$ for every $n\ge 1$, and
  \item $\alpha_n\le R_n$ for every $n\ge 1$.
\end{enumerate}
In the first case Kershner and Wintner \cite{KW}*{Pages 543--544} showed that the support of $\rho_\alpha$ is nowhere dense (and so is a Cantor-type set) and has total length $\ell=2\lim_{n\to\infty}2^nR_n$. The measure $\rho_\alpha$ is singular if and only if $\ell=0$; if $\ell>0$ then
\begin{equation*}
  \rho_\alpha(I)=\frac1\ell \Leb(I\cap S_\alpha)
\end{equation*}
where $\Leb$ is the usual Lebesgue measure on the real line and $S_\alpha$ is the spectrum.

In Case 2 the spectrum is all of $[-R_0,R_0]$ \cite{KW}*{Page 547}, and, as the examples below illustrate, the measure $\rho_\alpha$ may be absolutely continuous or singular.

We now list some examples. One particularly simple choice is $\alpha_n = a^n$ for some $a \in (0,1)$ and we write $\rho_a$ and $r_a$ for the corresponding spectral measure and covariance function.
\begin{itemize}
  \item If $0<a<\frac12$ we are in Case 1 and $\rho_a$ is singular. For rational $a$, by \cite{Ker}, $r_a(t)\to 0$ as $|t|\to \infty$ if and only if $\frac1a$ is not an integer. In this case we have $r_{a}(t) = O(|\log t|^{-\gamma})$ for large $|t|$ where $\gamma=\gamma(a)>0$.

      In the particular case $a=\frac13$ we get the usual Cantor middle-third set (shifted to be contained in $[-\frac12,\frac12]$) and the distribution function $\rho_{1/3}((-\infty,x])$ is the usual devil's staircase function (again, shifted).

  \item If $a\ge \frac12$ then we are in Case 2. If $a=\frac12$ then $r_{1/2}(t)=\sinc t$ (this formula was discovered by Euler), which was covered in Section~\ref{sec: exa square int}.

  \item More generally if $a=(\frac1{2})^{1/k}$ where $k\in \N$ then $\rho_a$ is absolutely continuous and $r_a(t)=O(|t|^{-k})$ (see \cite{Win}*{Page 836}). Once more we are covered by Section~\ref{sec: exa square int}.

  \item Erd\H{o}s \cite{Er}*{Section 2} showed that if $a=\frac1b$ where $b\neq 2$ is a `Pisot-Vijayaraghavan number' (a real algebraic integer whose conjugates lie in the unit disc) then  $\limsup_{t \rightarrow \infty} r_{a}(t) > 0$ and therefore $\rho_a$ is singular. Concrete values of $a>\frac12$ are $a=\frac{\sqrt5 -1}2$ (the Fibonacci number) and the positive root of the cubic $a^3+a^2-1$.

  \item Conversely, Salem \cite{Sa}*{Theorem II} showed that if $r_{a}(t)\not\to0$ then $\frac1a$ is a Pisot-Vijayaraghavan number. This, of course, does not rule out the possibility that there are other values of $a>\frac12$ for which $\rho_a$ is singular but $r_{a}(t)\to0$, but to the best of our knowledge there has been no progress on this question since.
\end{itemize}
More involved examples give more sophisticated behaviour, unless otherwise indicated the examples are taken from \cite{JW}*{Examples 1--8}.
\begin{itemize}
  \item If $\alpha_{2n-1}=\alpha_{2n}=\frac1{3^n}$ then $\rho_\alpha=\rho_{1/3} * \rho_{1/3}$ which is supported on the whole interval $[-1,1]$ and $r_\alpha(t)=r_{1/3}(t)^2$. We therefore have $\limsup_{t \rightarrow \infty} r_{\alpha}(t) > 0$ and so the spectral measure is singular.

  \item If the sequence $\alpha_n$ consists of the numbers of the form $2^{-m!}$ repeated exactly $2^{m!}$ times for $m=1,2,\dots$ then the spectrum is all of $\R$ and $r_{\alpha}(t)\not\to0$, which means that $\rho_\alpha$ is singular.

  \item If $\alpha_n=\frac{1}{n2^n}-\frac{1}{(n+1)2^{n+1}}$ then $\rho_\alpha$ is singular but $r_{\alpha}(t)\to0$.

  \item If $\alpha_n=\frac1{n!}$ then $\alpha_n>R_n$ for every $n$ and $(-1)^{k+1}r_\alpha(\pi k!)\to 1$ as $k\to\infty$, by\footnote{It is actually stated there that $(-1)^{k}r_\alpha(\pi k!)\to 1$ but this is easily seen to be incorrect.} \cite{LR}*{Proposition 9}.
\end{itemize}

We now present a general result which applies in Case 1. Its proof appears in Section~\ref{sec: cont sing}.
\begin{prop}\label{prop: sing var}
  Suppose that $\alpha_n>R_n$ for every $n\ge 1$, and define $M=M_T$ by $R_M<\frac1T\le R_{M-1}$. Then
  \begin{equation*}
    T\int_{0}^T \left(1-\frac{t}{T}\right) \left(r(t)+\frac{r''(t)}{\si^2}\right)^2 \ dt  \asymp T^2 2^{-M}
  \end{equation*}
  for $T\ge T_0$.
\end{prop}
\begin{rmk}
  This is of course a lower bound for $\var[N(\rho_\alpha;T)]$ in general, and gives the correct order of growth when $r_\alpha(t)\to 0$.
\end{rmk}
Let us apply this result to some of the examples above.
\begin{itemize}
  \item If $\alpha_n=a^n$ where $0<a<\frac12$ then $T^2 2^{-M}\asymp T^{2-d}$ where $d=\frac{\log 2}{\log \frac1a}$ is the Hausdorff dimension of the spectrum. It would be interesting to understand if there is any relation between the dimension of the spectrum and the behaviour of the variance for a general singular measure. (See also the remarks after Lemma~\ref{lem: cont sing measure bound}.)

  \item If $\alpha_n=\frac1{n!}$ then it is not difficult to show that $\alpha_n>R_n$ and $T^2 2^{-M}\asymp T^{2-\frac{\log2+o(1)}{\log\log T}}$. A less precise lower bound for the variance of this process is claimed in \cite{LR}*{Proposition 11}, although the proof given there is not entirely convincing. Nonetheless, we do use some of the ideas found there in our proof of Proposition~\ref{prop: sing var}.

  \item By choosing a sequence $R_n$ decaying sufficiently fast, it is clear that we can make the term $2^{-M}$ decay arbitrarily slowly. We can therefore construct a process with non-atomic spectral measure whose variance grows faster than $T^{2-\phi(T)}$  where $\phi(T) \to 0$ arbitrarily slowly, that is, arbitrarily close to maximal growth.
\end{itemize}

\subsection{Cancellation in the quantity \texorpdfstring{$r + \frac{r''}{\si^2}$}{}}\label{sec: Cancellation}
As we indicated previously, an important message of this paper is that the behaviour of the variance is governed by the quantity $r + \frac{r''}{\si^2}$. We wish to emphasise the important r\^{o}le of cancellation between the two terms here, and we have already presented some examples of this when the spectral measure has an atom at a `special frequency'. However this cancellation phenomenon is not just about atoms, and as an illustrative example we will produce a\footnote{In fact we produce a family of such covariance functions.} covariance function $r$ such that:
\begin{itemize}
\item The spectral measure $\rho$ has an $\mathcal{L}^1(\R)$ density.
\item $r + \frac{r''}{\si^2} \in \mathcal{L}^2(\R)$ where $\si^2 = \int_{\R} \lambda^2 \ d\rho(\lambda)$.
\item $r, r'' \notin \mathcal{L}^2(\R)$.
\end{itemize}
Writing $d\rho(\lm)=\phi(\lm)d\lm$ and applying the Fourier transform we see that it is equivalent to produce a function $\phi\ge0$ satisfying:
\begin{enumerate}[label=\arabic*.]
\item $\int_{\R}\phi(\lm) d\lm=1$ but $\phi \notin \mathcal{L}^2(\R)$.
\item $\lambda^2 \phi(\lambda) \in \mathcal{L}^1(\R)$, but $\lambda^2 \phi(\lambda) \notin \mathcal{L}^2(\R)$.
\item $\left(1 - \frac{\lambda^2}{\si^2} \right) \phi(\lambda) \in \mathcal{L}^2(\R)$ where $\sigma^2 = \int_{\R} \lambda^2 \phi(\lm) d\lambda$.
\end{enumerate}
We proceed to produce such a function $\phi$.

Let $\alpha \in \left(\frac{1}{2}, 1 \right)$. Choose $M > 1$ such that
\begin{equation} \label{eq : inequality for positivity in example}
M^2 + M + 1 > 3 + 3(1-\alpha) \left( \frac{1}{3-\alpha} - \frac{2}{2-\alpha} \right),
\end{equation}
and let $c_1, c_2 \in \R$ be the solution of the linear system
\begin{equation}
\label{eq : cancellation condition example}
\begin{array}{ccccccc}
\frac{1}{1-\alpha}  & c_1 & + & (M-1) & c_2 & = & \frac{1}{2}, \\
\left(\frac{1}{1-\alpha} - \frac{2}{2-\alpha} + \frac{1}{3-\alpha} \right) & c_1 & + & \frac{M^3 - 1}{3} & c_2  & = & \frac{1}{2}.
\end{array}
\end{equation}
We note that  \eqref{eq : inequality for positivity in example} ensures that the determinant of the matrix associated to \eqref{eq : cancellation condition example} is positive, and since we also have
$ \frac{M^3 - 1}{3} > M - 1$ and $ \frac{2}{2-\alpha} > \frac{1}{3-\alpha}$, it follows that $c_1, c_2 > 0$.
Define
$$
\phi(\lambda) =
\begin{cases}
c_1 (1 - |\lambda|)^{-\alpha}, & \text{ for } |\lambda| < 1, \\
c_2, &  \text{ for } 1 < |\lambda| < M.
\end{cases}
$$
Then:
\begin{itemize}
\item Since $\alpha \in \left(\frac{1}{2}, 1\right)$, it follows that $\phi \in \mathcal{L}^1(\R)$ but $\phi \notin \mathcal{L}^2(\R)$.
\item Integration yields, by the first equation in \eqref{eq : cancellation condition example}, that $\int_{\R}\phi(\lm) d\lm=1$.
\item Similarly $\lambda^2 \phi(\lambda) \in \mathcal{L}^1(\R)$, but $\lambda^2 \phi(\lambda) \notin \mathcal{L}^2(\R)$.
\item Now the second equation in \eqref{eq : cancellation condition example} shows that $\sigma^2 = \int_{\R} \lambda^2 \phi(\lm) d\lambda = 1$.
\item Finally note that $\left(1 - \lambda^2\right) \phi(\lambda) \in \mathcal{L}^2(\R)$.
\end{itemize}

\section{A formula for the variance}\label{sec: chaos expan}

The goal of this section is to give an infinite series expansion for $\var[N(T)]$, each coming from a different component of the Wiener chaos (or Hermite-It\^{o}) expansion of $N(T)$. We begin with some notation. For $q\in \N$ and $l, l_1, l_2, n \in \N_0 =\N\cup\{0\}$ write\footnote{We adopt the standard convention $\frac1{n!}=0$ when $n$ is a negative integer.}
\begin{equation}\label{eq : a coeffs}
a_{q}(l) = \frac{1}{l! (q-l)!} \cdot\frac{1}{2l-1}
\end{equation}
and
\begin{equation}\label{eq : b coeffs}
b_{q}(l_{1},l_{2},n)=\frac{(2q-2l_{1})!(2l_{1})!(2q-2l_{2})!(2l_{2})!}{(2q-2l_{1}-2l_{2}+n)!(2l_{1}-n)!(2l_{2}-n)!n!}.
\end{equation}
Next define the polynomials
\begin{equation}\label{eq : polys no correction}
\widetilde{P}_{q}(x,y,z)= \sum_{l_{1},l_{2}=0}^{q}a_{q}(l_{1}) a_q(l_{2})\sum_{n=\max(0,2(l_{1}+l_{2}-q))}^{\min(2l_{1},2l_{2})}b_{q}(l_{1},l_{2},n)\cdot x^{2(q-l_{1}-l_{2})+n}y^{2(l_{1}+l_{2}-n)}z^{n}
\end{equation}
and
\begin{equation}\label{eq : polys with correction}
P_{q}(x,y,z) =\widetilde{P}_{q}(x,y,z) + c_{q}\left(x^{2q-1}z+(2q-1)x^{2q-2}y^{2}\right)
\end{equation}
where
\begin{equation}\label{eq : c coeffs}
c_{q}=\frac{2^{4q}(q!)^{2}}{2q(2q)!} = \frac{2^{4q}}{2q \binom{2q}{q}}.
\end{equation}
We are now ready to state the expansion.
\begin{prop}\label{prop: main}
We have
\begin{equation*}
\var N(T) = \frac{\sigma^2}{\pi^2} \sum_{q=1}^{\infty} \frac{V_q(T)}{4^q} + \frac{\arccos r(T)}{\pi} \left(1-\frac{\arccos r(T)}{\pi} \right)
\end{equation*}
where
\begin{equation}\label{eq : V_q(T)}
V_q(T) = 2\int_{0}^{T} \left( T - t \right) P_q \left(  r(t), \frac{r'(t)}{\sigma}, \frac{r''(t)}{\sigma^2} \right) dt.
\end{equation}
Furthermore
\begin{equation}\label{eq: Lb for N}
\var N(T) \ge \frac{\sigma^2}{4\pi^2} V_1(T) +  \frac1{\pi^2}\left(1-r(T)^2 \right).
\end{equation}
\end{prop}

The starting point in our calculations is the following Hermite expansion for $N(T)$ given by Kratz and L\'eon \cite{KL10}*{Proposition 1} assuming only the Geman condition (though they and other authors had considered it previously under more restrictive assumptions). We have (the sum converges in $L^2(\Pro)$)
\begin{equation*}
N(T) =\frac{\si}{\pi} \sum_{q=0}^\infty \frac{(-1)^{q+1}}{2^q} N_q(T)
\end{equation*}
where\footnote{Under the Geman condition, one cannot assume that $f$ is continuously differentiable, and `conversely' a continuously differentiable process need not satisfy the Geman condition, see \cite{Gem}*{Section 4}. However the existence of $r''$ implies the existence of the derivative in quadratic mean of the process, and this is how the object $f'$ should be understood if the process is not differentiable.}
\begin{equation}\label{eq: Nq}
N_q(T) = \sum_{l=0}^q a_q(l)\int_0^T H_{2(q-l)}(f(t)) H_{2l} (f'(t)/\si) \ dt,
\end{equation}
and $H_l$ is the $l$'th Hermite polynomial. Further each $N_q(T)$ belongs to the $2q$'th Wiener chaos which yields
\[
\E[N(T)] = \frac{\si}{\pi} N_0(T) = \frac{\si}{\pi}T,
\]
and
\begin{equation} \label{eq : variance chaos decomposition}
\var[N(T)] =\frac{\si^2}{\pi^2} \sum_{q=1}^\infty 4^{-q} \E[N_q(T)^2].
\end{equation}
Furthermore
\begin{equation}\label{eq: var bigger first}
  \var[N(T)]\ge \frac{\si^2}{\pi^2} \frac{\E[N_1(T)^2]}4.
\end{equation}
The next lemma allows us to evaluate $\E \left[ N_q(T)^2 \right]$
\begin{lem}\label{lem: ENq^2}
For all $q \in \N$
\begin{equation*}
\E \left[ N_q(T)^2 \right] = 2\int_{0}^{T} \left( T - t \right) \widetilde{P}_q \left(  r(t), \frac{r'(t)}{\sigma}, \frac{r''(t)}{\sigma^2} \right) dt,
\end{equation*}
where $\widetilde{P}_q$ is given by \eqref{eq : polys no correction}.
\end{lem}
We now show how this lemma yields the desired expression.

\begin{proof}[Proof of Proposition~\ref{prop: main}, assuming Lemma~\ref{lem: ENq^2}]
Lemma~\ref{lem: ENq^2} yields
\begin{equation*}
\E \left[ N_q(T)^2 \right] = V_q(T) - \frac{2c_q}{\sigma^2} \int_{0}^{T} \left(T - t \right)
\left( r(t)^{2q-1} r''(t) + (2q-1) r(t)^{2q-2} r'(t)^2 \right) \ dt.
\end{equation*}
Note that $ r(t)^{2q-1} r''(t) + (2q-1) r(t)^{2q-2} r'(t)^2 = \frac{d^2}{dt^2} \left[ \frac{r(t)^{2q}}{2q} \right]$ and so
\begin{align*}
\int_{0}^{T} \left(T - t \right)&
\Big( r(t)^{2q-1} r''(t) + (2q-1) r(t)^{2q-2} r'(t)^2 \Big) \ dt\\
& = \frac{1}{2q} \int_{0}^{T} (T - t) \frac{d^2}{dt^2} \left[ r(t)^{2q}\right] \ dt\\
& = \frac{1}{2q} \left[ (T-t) \cdot 2q \cdot r(t)^{2q-1} r'(t) \Big \vert_{t=0}^{T} + \int_{0}^{T} \frac{d}{dt} \left[ r(t)^{2q} \right]  \ dt \right]\\
&=  \frac{1}{2q} \left[ r(T)^{2q} - 1 \right].
\end{align*}
We therefore have
$$
\E \left[ N_q(T)^2 \right] = V_q(T) + \frac{c_q}{q \sigma^2} \left(1 - r(T)^{2q} \right).
$$
Applying \eqref{eq: var bigger first} yields the desired lower bound
\begin{equation*}
  \var[N(T)]\ge \frac{\si^2}{4\pi^2}\E[N_1(T)^2]= \frac{\si^2}{4\pi^2}V_1(T)+\frac{1}{\pi^2}(1-r(T)^2)
\end{equation*}
while \eqref{eq : variance chaos decomposition} gives
\begin{align*}
\var \left[ N(T) \right] &=
\frac{\sigma^2}{\pi^2} \sum_{q=1}^{\infty} \frac{1}{4^q} \left[ V_q(T) + \frac{c_q}{q \sigma^2} \left(1 - r(T)^{2q} \right)  \right]  \\
& =  \frac{\sigma^2}{\pi^2} \sum_{q=1}^{\infty} \frac{V_q(T)}{4^q}  +
\frac{1}{2\pi^2} \sum_{q=1}^{\infty} \frac{2^{2q} - (2r(T))^{2q}}{q^2 {\binom{2q}{q}}}.
\end{align*}
We identify the last series as
\begin{equation}\label{eq: arcsin taylor}
  \arcsin^2(x) = \frac 1 2  \sum_{q=1}^\infty \frac{2^{2q}}{q^2 \binom{2q}{q}} x^{2q}
\end{equation}
for all $|x|\le 1$ implying that
\begin{align*}
\var [ N(T) ] & = \frac{\sigma^2}{\pi^2} \sum_{q=1}^{\infty} \frac{V_q(T)}{4^q}  + \frac{\arcsin^2(1) - \arcsin^2(r(T)) }{\pi^2}  \\
& = \frac{\sigma^2}{\pi^2} \sum_{q=1}^{\infty} \frac{V_q(T)}{4^q}  +  \frac{\arccos r(T)}{\pi} \left(1-\frac{\arccos r(T)}{\pi} \right),
\end{align*}
where the last equality follows from $\arccos(x) = \frac{\pi}{2} - \arcsin(x)$.
\end{proof}

We now proceed to prove Lemma~\ref{lem: ENq^2}.

\begin{proof}[Proof of Lemma~\ref{lem: ENq^2}]

Squaring the expression for $N_q(T)$ given in \eqref{eq: Nq} yields
\begin{equation*}
  N_q(T)^2 =  \sum_{l_1, l_2=0}^q  a_q(l_1) a_q(l_2)
\int_0^T \int_{0}^{T} H_{2(q-l_1)}(f(t)) H_{2(q-l_2)}(f(s)) H_{2l_1} \left( \frac{f'(t)}{\sigma} \right) H_{2l_2} \left( \frac{f'(s)}{\sigma} \right) ds \,dt.
\end{equation*}
and so
\begin{align*}
  \E \big[ &N_q(T)^2 \big]\\
   & =\sum_{l_1, l_2=0}^q  a_q(l_1) a_q(l_2) \int_0^T \int_{0}^{T} \E \left[ H_{2(q-l_1)}(f(t)) H_{2(q-l_2)}(f(s)) H_{2l_1} \left( \frac{f'(t)}{\sigma} \right) H_{2l_2} \left( \frac{f'(s)}{\sigma} \right) \right] ds\, dt.
\end{align*}
Applying Lemma~\ref{lem : diagram chaos contribution} below, and using the simple change of variables
\[
\int_0^T \int_0^T h(t-s) dt \, ds = \int_{-T}^T (T-|x|) h(x) dx
\]
for any $h\in L^1([-T,T])$, we get
\begin{equation*}
  \E \big[ N_q(T)^2 \big]=\int_{-T}^{T} \left( T - |t| \right) \widetilde{P}_q \left(  r(t), \frac{r'(t)}{\sigma}, \frac{r''(t)}{\sigma^2} \right) dt.
\end{equation*}
Noting that $r$ is an even function and that only even powers of $y$ appear in $\widetilde{P}_q$ yields Lemma~\ref{lem: ENq^2}.
\end{proof}
\begin{lem}\label{lem : diagram chaos contribution}
For all $q\in \N$ and $l_1,l_2 \in \N_0$ such that $0 \le l_1, l_2 \le q$ we have
\begin{align*}
\E \bigg[& H_{2q-2l_1} \left( f(t) \right) H_{2q-2l_2} \left( f(s) \right) H_{2l_1} \left( \frac{f'(t)}{\sigma} \right) H_{2l_2} \left( \frac{f'(s)}{\sigma} \right) \bigg]\\
&= \sum_{n = \max(0, 2l_1 + 2l_2 - 2q)}^{\min(2l_1, 2l_2)}  b_q(l_1, l_2, n)
 \left( \frac{r''(t-s)}{\sigma^2} \right)^n \left( \frac{r'(t-s)}{\sigma} \right)^{2(l_1 + l_2 - n)} \left( r(t-s) \right)^{2(q-l_1-l_2)+n}.
\end{align*}
\end{lem}

Before proving the lemma we first recall the diagram formula.
\begin{lem}[The diagram formula \citelist{\cite{BM}*{Page 432} \cite{J}*{Theorem 1.36}}]
Let $X_1, \ldots, X_k$ be jointly Gaussian random variables, and $n_1, \ldots, n_k \in \N$.
A Feynman diagram is a graph with $n_1 + \ldots + n_k$ vertices such that
\begin{itemize}
\item There are $n_i$ vertices labelled $X_i$ for each $i$ (and each vertex has a single label). For a vertex $a$ we write $X_{\ell(a)}$ for the label of $a$.
\item Each vertex has degree $1$.
\item No edge joins $2$ vertices with the same label.
\end{itemize}
Let $\mathscr{D}$ be the set of such diagrams. For $\gamma \in \mathscr{D}$ we define the value of $\gamma$ to be
$$
v(\gamma) = \prod_{(a,b) \in E(\gamma)} \E \left[ X_{\ell(a)} X_{\ell(b)} \right]
$$
where $E(\gamma)$ is the set of edges of $\gamma$.
Then
$$
\E \left[ H_{n_1}(X_1) \cdot \cdots \cdot H_{n_k} (X_k) \right] = \sum_{\gamma \in \mathscr{D}} v(\gamma).
$$
\end{lem}

\begin{proof}[Proof of Lemma~\ref{lem : diagram chaos contribution}]
We apply the diagram formula to the random variables $f(t),f(s),f'(t) / \sigma$ and $f'(s) / \sigma$ and corresponding integers $2(q-l_1), 2(q-l_2), 2l_1$ and $2l_2$ and denote by $\mathscr{D}$ the collection of relevant Feynman diagrams. Since $\E \left[f(t) f'(t) \right] = \E \left[ f(s) f'(s) \right] = r'(0) = 0$, it is enough to consider diagrams whose edges do not join vertices labeled $f(t) $ to $f'(t)/ \sigma$ or vertices labeled $f(s)$ to $f'(s) / \sigma$.

\begin{figure}
  \centering
  \includegraphics[clip, trim=1cm 11cm 2cm 2cm,  width=0.4\columnwidth]{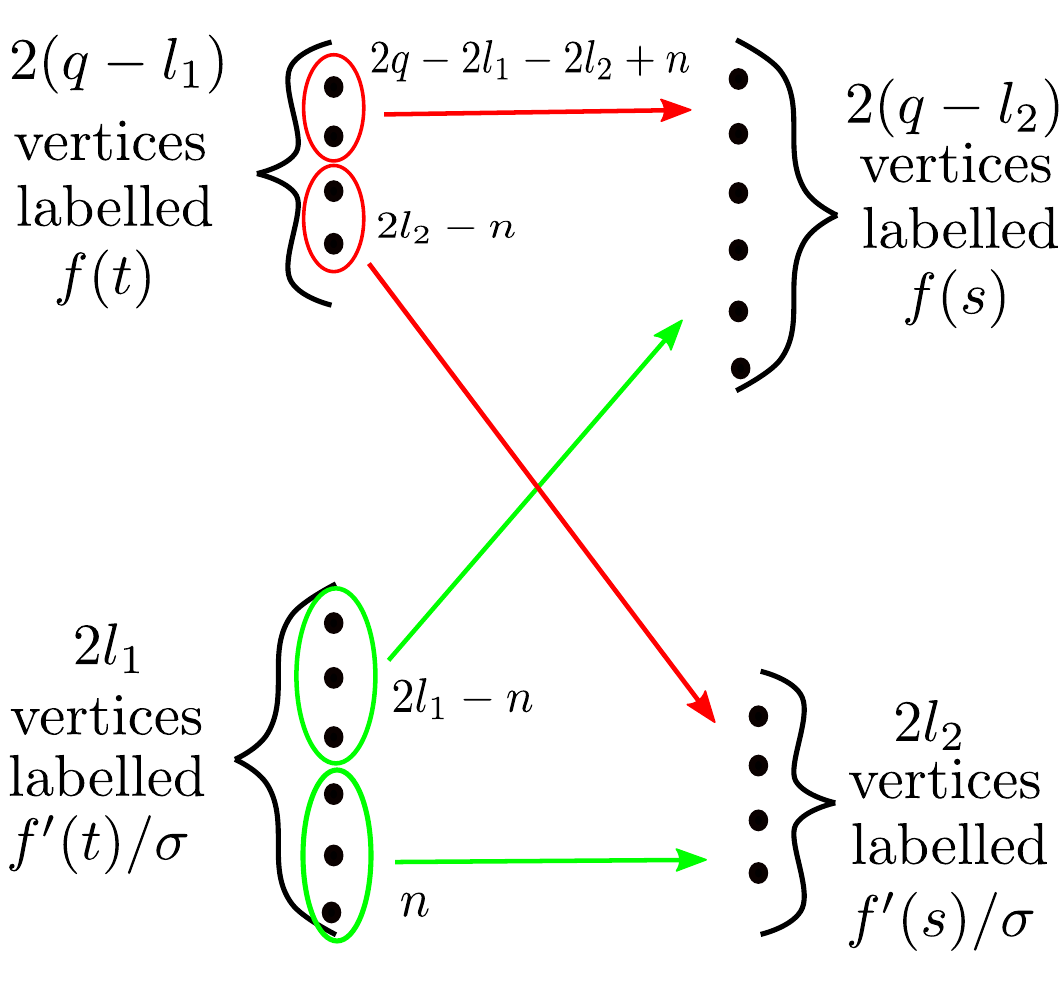}
  \caption{Counting the number of Feynman diagrams}
  \label{fig: count diag}
\end{figure}

Let $n$ be the number of edges joining a vertex labeled $f'(t) / \sigma$ to a vertex labeled $f'(s) / \sigma$, see Figure~\ref{fig: count diag}. Then $0 \le n \le \min(2l_1, 2l_2)$. Moreover, as the other vertices labeled $f'(t)/ \sigma$ must be joined to vertices labeled $f(s)$, we see that $2l_1 - n \le 2q - 2l_2$, so $\max(0, 2l_1 + 2l_2 - 2q) \le n \le \min(2l_1, 2l_2)$. Further, every value of $n$ in this range is attained by some diagram.

We compute the value of such a diagram to be
\begin{align*}
v(\gamma) &= \E \left[ f'(t)f'(s) / \sigma^2 \right]^{n} \E \left[ f'(t) f(s) / \sigma \right]^{2l_1 - n}
\E \left[ f(t) f'(s) / \sigma \right]^{2l_2-n} \E \left[ f(t) f(s) \right]^{2q-2l_1-2l_2+n}  \\
&= \left( \frac{r''(t-s)}{\sigma^2} \right)^n \left( \frac{r'(t-s)}{\sigma} \right)^{2(l_1 + l_2 - n)}  \left(r(t-s)\right)^{2(q-l_1-l_2)+n} .
\end{align*}
Finally, we count the number of such diagrams. There are
$$
\binom{2l_1}{n} \binom{2l_2}{n} n!
$$
ways to choose $n$ vertices labeled $f'(t) / \sigma$, to choose $n$ vertices labeled $f'(s) / \sigma$ and to pair them. There are
$$
\binom{2q-2l_2}{2l_1-n} (2l_1-n)!
$$
ways to choose $2l_1-n$ vertices labeled $f(s)$ and to pair them with the remaining vertices labeled $f'(t) / \sigma$. There are
$$
\binom{2q-2l_1}{2q-2l_1-2l_2+n} (2q-2l_1-2l_2+n)!
$$
ways to choose $2q-2l_1-2l_2+n$ vertices labeled $f(t)$ and to pair them with the remaining ones labeled $f(s)$. There are $$(2l_2-n)!$$ ways to pair the remaining vertices labeled $f(t)$ and $f'(s) / \sigma$. Since these choices are independent, we multiply these counts to get that there are $b_q(l_1,l_2,n)$ such diagrams, where $b_q$ is given by \eqref{eq : b coeffs}.
Applying the diagram formula completes the proof.
\end{proof}

\section{Lower bound}\label{sec: LB}
In this section we prove Proposition~\ref{prop: LB}. From Proposition~\ref{prop: main} we have
\begin{equation*}
  \var[N(T)]\ge \frac{\sigma^2}{4\pi^2} V_1(T)
\end{equation*}
and the first statement of Proposition~\ref{prop: LB} follows simply by computing
\begin{equation*}
P_1(x,y,z) = 2(x+z)^2
\end{equation*}
which gives
\begin{equation}\label{eq: V_1}
  V_1(T) = 4T\int_{0}^T \left(1-\frac{t}{T}\right) \left(r(t)+\frac{r''(t)}{\si^2}\right)^2 dt.
\end{equation}

To deduce the second statement it is enough to find an interval $I$ such that $\left | r+\frac{r''}{\si^2} \right | \ge C>0$ on $I$. But this follows from the fact that $r''$ is continuous and $r$ is not cosine.

\section{Upper bound}
In this section, we prove Theorems~\ref{thm: UB} and \ref{thm : asymp constant}. Our method is to bound each $V_q(T)$ by $V_1(T)$ and apply Proposition~\ref{prop: main}. We achieve this by proving the following properties of the polynomials $P_q$ (recall \eqref{eq : polys with correction}).

\begin{prop} \label{prop: divisibility}
For all $q\ge 1$ we have $(x+z)^{2}\mid P_{q}(x,y,z)$.
\end{prop}

\begin{prop}\label{prop: Crude Arcones-type bound}
Set $M = \max(|x| + |y|, |y| + |z|)$. Then
\begin{equation}
\frac{|P_q(x,y,z)|}{(x+z)^2} \le \frac{e^2}{\sqrt{\pi }} q^{3/2} 4^{q} M^{2q-2}.
\end{equation}
\end{prop}
Proving Proposition~\ref{prop: divisibility} amounts to proving some identities for the coefficients of the polynomials $P_q$, which is deferred to Section~\ref{section: proof of divisibility} where we implement a general method due to Zeilberger \cite{AZ06}. We proceed to prove Proposition~\ref{prop: Crude Arcones-type bound}.

\subsection{Proof of Proposition~\ref{prop: Crude Arcones-type bound}}
\label{section: Arcones}
By Proposition~\ref{prop: divisibility}, we may prove Proposition~\ref{prop: Crude Arcones-type bound} by bounding the second derivative of $P_q$. To achieve this we borrow the main idea from the proof of Arcones' Lemma~\cite{A}*{Lemma 1}.

\begin{proof}[Proof of Proposition~\ref{prop: Crude Arcones-type bound}]
Our goal is to bound $\frac{\partial^2 P_q}{\partial x^2}$. For $k \le 2q-2$, define
$$
\alpha_q(k) = \begin{cases}
0, & \text{for odd } k, \\
\frac{1}{q!} \cdot \binom{q}{k/2} \frac{(2q-k)!k!}{k-1}, & \text{for even }k,
\end{cases}
$$
which yields (recall \eqref{eq : a coeffs})
\begin{equation*}
\alpha_q(2k) = \binom{q}{k} \frac{(2q-2k)!(2k)!}{(2k-1)\cdot q!} = (2q-2k)!(2k)! \cdot a_q(k).
\end{equation*}
Let $0\le k,l \le 2q-2$ and suppose that $n$ is an integer such that $\max(0, l+k-2q+2) \le n \le \min(l,k)$. Recalling \eqref{eq : b coeffs} we have
\begin{align*}
\alpha_q(2k) \alpha_q(2l) &= (2q-2k)! (2k)! (2q-2l)! (2l)! \cdot a_q(k) a_q(l) \\
&= (2q-2k-2l+n)! (2k-n)!(2l-n)! n! \cdot a_q(k) a_q(l) b_q(k,l,n)
\end{align*}
and so
\begin{equation}\label{eq: second derivative}
a_q(k)a_q(l) b_q(k,l,n)  \frac{\partial^2 }{\partial x^2}\left[x^{2(q-k-l)+n}\right]
= \frac{\alpha_q(2k) \alpha_q(2l) x^{2q-2k-2l-2+n}}{(2q-2k-2l-2+n)!(2k-n)!(2l-n)!n!}.
\end{equation}

Let
\begin{equation*}
k_1 = 2q-2-k, \quad k_2 = k,\quad l_1 = 2q-2-l,\quad\text{and}\quad l_2 = l,
\end{equation*}
define
\begin{align*}
\mathcal{A}(k,l) &= \left \{  \left( \begin{array}{cc}
2q-l-k-2+n & l-n \\
k-n & n
\end{array}
\right) :
\max(0, l+k-2q+2) \le n \le \min(l,k)
 \right \} \\
 &= \left \{ a = \left( \begin{array}{cc}
 a_{11} & a_{12} \\
 a_{21} & a_{22}
 \end{array}
  \right) : a_{ij} \in \N_{0}, a_{i1} + a_{i2} = k_{i}, a_{1i} + a_{2i} = l_{i}
  \right \}
\end{align*}
and
\begin{equation*}
  \tilde{\mathcal{A}}(k)=\bigcup_{l=0}^{2q-2} \mathcal{A}(k,l) = \left \{ a = \left( \begin{array}{cc}
 a_{11} & a_{12} \\
 a_{21} & a_{22}
 \end{array}
  \right) : a_{ij} \in \N_{0}, a_{i1} + a_{i2} = k_{i}
  \right \}.
\end{equation*}
Then, using \eqref{eq: second derivative} and recalling \eqref{eq : polys no correction}, we have
\begin{align*}
\frac{\partial^2 \widetilde{P}_q}{\partial x^2} &= \sum_{k,l=0}^{q-1} \alpha_q(2k) \alpha_q(2l) \sum_{n=\max(0,2k+2l-2q+2)}^{\min(2k,2l)}
 \frac{x^{2(q-k-l-1)+n}}{(2q-2k-2l-2+n)!} \frac{y^{2k-n}}{(2k-n)!} \frac{y^{2l-n}}{(2l-n)!} \frac{z^{n}}{n!} \\
 &= \sum_{k,l=0}^{q-1} \alpha_q(2k) \alpha_q(2l) \sum_{A \in \mathcal{A}(2k,2l)} \prod_{i,j=1}^{2} \frac{x_{ij}^{a_{ij}}}{a_{ij}!}\\
 & = \sum_{k,l=0}^{2q-2} \alpha_q(k) \alpha_q(l) \sum_{A \in \mathcal{A}(k,l)} \prod_{i,j=1}^{2} \frac{x_{ij}^{a_{ij}}}{a_{ij}!}
\end{align*}
where
$$
x_{11} = x, \quad x_{12} = x_{21} = y, \quad\text{and}\quad x_{22} = z.
$$

We now bound
\begin{align*}
\left | \frac{\partial^2 \widetilde{P}_q}{\partial x^2} \right | &
\le
 \sum_{k,l=0}^{2q-2} \left | \alpha_q(k) \alpha_q(l) \right | \sum_{a \in \mathcal{A}(k,l)} \prod_{i,j=1}^{2} \frac{|x_{ij}|^{a_{ij}}}{a_{ij}!} \\
 &\le  \sum_{k,l=0}^{2q-2} \left( \frac{\alpha_q(k)^2 + \alpha_q(l)^2}{2} \right) \sum_{a \in \mathcal{A}(k,l)} \prod_{i,j=1}^{2} \frac{|x_{ij}|^{a_{ij}}}{a_{ij}!}
\end{align*}
Algebraic manipulation of this last quantity yields
\begin{equation*}
  \left | \frac{\partial^2 \widetilde{P}_q}{\partial x^2} \right |
\le \sum_{k=0}^{2q-2} \alpha_q(k)^2 \sum_{l=0}^{2q-2} \sum_{a \in \mathcal{A}(k,l)} \prod_{i,j=1}^{2} \frac{|x_{ij}|^{a_{ij}}}{a_{ij}!}
 =  \sum_{k=0}^{2q-2} \alpha_q(k)^2 \sum_{a \in\tilde{\mathcal{A}}(k)} \prod_{i,j=1}^{2} \frac{|x_{ij}|^{a_{ij}}}{a_{ij}!}
\end{equation*}
Applying the Binomial Theorem to the last term gives
\begin{align*}
 \sum_{k=0}^{2q-2} \alpha_q(k)^2 \prod_{i=1}^{2} \frac{(|x_{i1}| + |x_{i2}|)^{k_i}}{k_i!}
 &\le M^{2q-2} \sum_{k=0}^{2q-2} \frac{\alpha_q(k)^2}{k! (2q-2-k)!}\\
 &\le 4q^2 M^{2q-2} \sum_{k=0}^{2q-2} \frac{\alpha_q(k)^2}{k! (2q-k)!}\\
 &= 4q^2 M^{2q-2} \sum_{k=0}^{q-1} \frac{1}{(q!)^2} {\binom{q}{k}}^2 \frac{(2k)!(2q-2k)!}{(2k-1)^2} = 4q^2 c_q M^{2q-2}.
\end{align*}
where the last identity is due to Lemma~\ref{lem: c_q} below, and we remind the reader of \eqref{eq : c coeffs}. We also have, from \eqref{eq : polys with correction}, that
\begin{equation*}
\frac{\partial^2 P_q}{\partial x^2} = \frac{\partial^2 \widetilde{P}_q}{\partial x^2} + (2q-1)(2q-2)c_q \left(x^{2q-3}z + (2q-3)x^{2q-4}y^2 \right).
\end{equation*}
We next bound this final summand. Note that for $q=1$ this term vanishes. Otherwise, on the domain $D_M = \{ |x| + |y| \le M, |y| + |z| \le M \}$, it attains its maximum on the boundary, and a calculation reveals the maximum is attained at $|z| = |x| = M, y = 0$. Therefore
\begin{equation*}
  | x^{2q-3}z + (2q-3)x^{2q-4}y^2| \le M^{2q-2}.
\end{equation*}

Combining these two estimates we obtain
\begin{equation*}
\left | \frac{\partial^2 P_q}{\partial x^2} \right | \le \left( 4q^2 + (2q-1)(2q-2) \right) c_q M^{2q-2}\le 8q^2 c_q M^{2q-2}.
\end{equation*}
Using Sterling's bounds \footnote{The constants here are not asymptotically optimal, but this is irrelevant for our purposes.} we see that $\binom{2q}{q} \ge\frac{2\sqrt\pi}{e^2}\frac{2^{2q}}{\sqrt q}$ which yields
$$
c_q = \frac{2^{4q}}{2q \binom{2q}{q}} \le \frac{e^2}{4\sqrt\pi}\frac{4^{q}}{\sqrt q}
$$
so that
\begin{equation} \label{eq : bound second derivative}
\sup_{D_M} \left | \frac{\partial^2 P_q}{\partial x^2}  \right | \le \frac{2e^2}{\sqrt\pi}q^{3/2}4^{q} M^{2q-2}.
\end{equation}

By the mean value theorem,
\[
P_q(x,y,z) = P_q(-z,y,z) + \frac{\partial P_q}{\partial x}(-z,y,z) (x+z) + \frac{1}{2}  \frac{\partial^2 P_q}{\partial x^2}(t,y,z) (x+z)^2
\]
for some $t$ between $x$ and $-z$.
It follows from Proposition~\ref{prop: divisibility} that $ P_q(-z,y,z) = \frac{\partial P_q}{\partial x}(-z,y,z)=0$, so that
\[
P_q(x,y,z)= \frac{1}{2}  \frac{\partial^2 P_q}{\partial x^2}(t,y,z) (x+z)^2.
\]
Note that $|t| \le \max(|x|, |z|) \le M - |y|$ and so by \eqref{eq : bound second derivative} we have
\begin{equation*}
\frac{|P_q(x,y,z)|}{(x+z)^2} \le \frac{1}{2} \sup_{(t,y,z)\in D_M } \left | \frac{\partial^2 P_q}{\partial x^2}(t,y,z) \right | \le \frac{e^2}{\sqrt\pi} q^{3/2} 4^{q} M^{2q-2}. \qedhere
\end{equation*}
\end{proof}

In the course of the proof we used the following computation.

\begin{lem}
\label {lem: c_q}
For all $q \in \N$ we have
$$
c_q = \sum_{l=0}^{q} \binom{2l}{l} \binom{2q-2l}{q-l} \frac{1}{(2l-1)^2}.
$$
\end{lem}

\begin{proof}
For $q \ge 0$, let us denote $T_q =  \sum_{l=0}^{q} \binom{2l}{l} \binom{2q-2l}{q-l} \frac{1}{(2l-1)^2}$. Notice that
\begin{equation}\label{eq: Tq sum}
\sum_{q=0}^{\infty} T_q x^{2q} = \phi(x) \psi(x)
\end{equation}
where
$$
\phi(x) = \sum_{l=0}^{\infty} \binom{2l}{l} \frac{x^{2l}}{(2l-1)^2}, \quad \text{and}\quad
\psi(x) =  \sum_{l=0}^{\infty} \binom{2l}{l} x^{2l} = \frac{1}{\sqrt{1-4x^2}} .
$$

We next compute $\phi$. We have
\begin{equation*}
\frac{d}{dx} \left[ \frac{\phi(x)}{x} \right] = \sum_{l=0}^{\infty} \binom{2l}{l} \frac{x^{2l-2}}{2l-1}
= -\frac{1}{x^2} \sqrt{1-4x^2}= \frac{d}{dx} \left[ \frac{\sqrt{1-4x^2}}{x} + 2\arcsin(2x) \right]
\end{equation*}
and so $\frac{\phi(x)}{x} = \frac{\sqrt{1-4x^2}}{x} + 2\arcsin(2x) + C$ for some constant $C$.
Since all the functions in this equation are odd, it follows that $C = 0$, and so $\phi(x) = \sqrt{1-4x^2} + 2x\arcsin(2x)$.
Therefore, using the Taylor series \eqref{eq: arcsin taylor} once more,
\begin{align*}
\phi(x)\psi(x) &= 1 + \frac{2x \arcsin(2x)}{\sqrt{1-4x^2}}\\
& = 1 + \frac{x}{2} \frac{d}{dx} \left( \arcsin(2x) \right)^2\\
&= 1 + \frac{x}{2} \frac{d}{dx} \sum_{q=1}^{\infty} \frac{(4x)^{2q}}{2q^2 \binom{2q}{q}} =
1 + \sum_{q=1}^{\infty} \frac{4^{2q} x^{2q}}{2q \binom{2q}{q}}.
\end{align*}
Comparing this with \eqref{eq: Tq sum} we conclude that $T_q = \frac{2^{4q}}{2q \binom{2q}{q}}=c_q$ for $q \ge 1$.
\end{proof}

\subsection{Proof of Theorem~\ref{thm: UB}}\label{section: proof of UB}
Having Proposition~\ref{prop: Crude Arcones-type bound} at our disposal, we are ready to prove Theorem~\ref{thm: UB}. Let
\begin{equation*}
  M' = \limsup_{|t| \rightarrow \infty} \varphi(t)<1
\end{equation*}
and choose $M \in \left( M', 1 \right)$. Then there exists some $T_0 > 0$ such that
$ \varphi(t) \le M $ for all $|t| > T_0$.
We can rearrange \eqref{eq : V_q(T)} to obtain
\begin{align}
\label{eq: V_q rearranged}
V_q(T)\nonumber =  V_q(T_0) &+  2\left( T - T_0 \right) \int_{0}^{T_0} P_q \left(r(t), \frac{r'(t)}{\sigma}, \frac{r''(t)}{\sigma^2} \right) \ dt \\
&\quad+ 2\int_{T_0}^T \left( T - t \right) P_q \left(  r(t), \frac{r'(t)}{\sigma}, \frac{r''(t)}{\sigma^2} \right) \ dt.
\end{align}
Proposition~\ref{prop: Crude Arcones-type bound} yields
\begin{align}
\nonumber
\left|\int_{T_0}^{T}  \left( T - t \right) P_q \left(  r(t), \frac{r'(t)}{\sigma}, \frac{r''(t)}{\sigma^2} \right) \ dt\right|
& \le \frac{e^2}{\sqrt{\pi }} q^{3/2} 4^{q} M^{2q-2} \int_{T_0}^{T} \left( T - t \right) \left( r(t) + \frac{r''(t)}{\sigma^2}  \right)^2  \ dt \\
\nonumber
& \le  \frac{e^2}{\sqrt{\pi }} q^{3/2} 4^{q} M^{2q-2} \int_{0}^{T} \left( T - t \right) \left( r(t) + \frac{r''(t)}{\sigma^2}  \right)^2  \ dt \\
\label{eq : tail upper bound single q}
& =  \frac{e^2}{\sqrt{\pi }} q^{3/2} 4^{q-1} M^{2q-2} V_1(T),
\end{align}
see \eqref{eq: V_1}. Since $M<1$ we see that
\begin{equation*}
  \sum_{q=1}^\infty\frac1{4^q}\left|\int_{T_0}^{T}  \left( T - t \right) P_q \left(  r(t), \frac{r'(t)}{\sigma}, \frac{r''(t)}{\sigma^2} \right) \ dt\right|<\infty.
\end{equation*}
By Proposition~\ref{prop: main}, since we are assuming the Geman condition, we have $\sum_{q=1}^{\infty} \frac{V_q(T)}{4^q} < \infty$ for every $T>0$ and so we may write, from \eqref{eq: V_q rearranged}
\begin{align*}
\sum_{q=1}^{\infty} \frac{V_q(T)}{4^q} &= \sum_{q=1}^{\infty} \frac{V_q(T_0)}{4^q} + \left( T - T_0 \right)  \sum_{q=1}^{\infty} \frac{1}{4^q} \int_{0}^{T_0} P_q \left(  r(t), \frac{r'(t)}{\sigma}, \frac{r''(t)}{\sigma^2} \right) \ dt \\
&\quad + \sum_{q=1}^{\infty} \frac{1}{4^q} \int_{T_0}^{T} \left( T - t \right) P_q \left(  r(t), \frac{r'(t)}{\sigma}, \frac{r''(t)}{\sigma^2} \right) \ dt.
\end{align*}
Combining this with \eqref{eq : tail upper bound single q} we get
\begin{equation*}
  \sum_{q=1}^{\infty} \frac{V_q(T)}{4^q}\le C_0 + C_1 T + C_2 V_1(T)
\end{equation*}
where $C_0,C_1$ and $C_2$ depend on $T_0$ and $M$. Recalling Proposition~\ref{prop: main} we have
\begin{equation*}
  \var [N(T)] \le \frac{\sigma^2}{\pi^2} \sum_{q=1}^{\infty} \frac{V_q(T)}{4^q} + \frac14 \le C_3 V_1(T)
\end{equation*}
where we have used the lower bound proved in Section~\ref{sec: LB} for the final bound.

\subsection{Proof of Theorem~\ref{thm : asymp constant}}
By \eqref{eq: V_1} we need to show that $\var [N(T)] \sim \frac{\sigma^2}{4\pi^2} V_1(T)$. The lower bound follows immediately from Proposition~\ref{prop: LB} and so we focus on the upper bound. We proceed as in the previous section, but estimate more carefully. By Proposition~\ref{prop: main} we have
\begin{equation*}
  \var [N(T)] \le \frac{\sigma^2}{4\pi^2} V_1(T) + \frac{\sigma^2}{\pi^2} \sum_{q=2}^{\infty} \frac{V_q(T)}{4^q} +\frac14.
\end{equation*}
Now fix $\ep>0$ and choose $T_0 = T_0(\ep)$ such that $\varphi(t) < \ep$ for all $t>T_0$. As in the previous section we write
\begin{align*}
\sum_{q=2}^{\infty} \frac{V_q(T)}{4^q} &= \sum_{q=2}^{\infty} \frac{V_q(T_0)}{4^q} + \left( T - T_0 \right)  \sum_{q=2}^{\infty} \frac{1}{4^q} \int_{0}^{T_0} P_q \left(  r(t), \frac{r'(t)}{\sigma}, \frac{r''(t)}{\sigma^2} \right) \ dt \\
&\quad + \sum_{q=2}^{\infty} \frac{1}{4^q} \int_{T_0}^{T} \left( T - t \right) P_q \left(  r(t), \frac{r'(t)}{\sigma}, \frac{r''(t)}{\sigma^2} \right) \ dt
\end{align*}
and estimate
\begin{equation*}
  \left|\int_{T_0}^{T}  \left( T - t \right) P_q \left(  r(t), \frac{r'(t)}{\sigma}, \frac{r''(t)}{\sigma^2} \right) \ dt\right| \le  \frac{e^2}{2\sqrt{\pi }} q^{3/2} 4^{q} \ep^{2q-2} V_1(T).
\end{equation*}
This yields
\begin{equation*}
  \sum_{q=2}^{\infty} \frac{V_q(T)}{4^q} =C_0 + C_1 T + \frac{2e^2}{\sqrt{\pi }} \sum_{q=2}^{\infty} q^{3/2} (4\ep^2)^{q-1} V_1(T)\le C_0 + C_1 T + C_3 \ep^2 V_1(T)
\end{equation*}
and we finally note that since $r+\frac{r''}{\si^2} \not\in \mathcal{L}^2(\R)$ we have
\begin{equation*}
  \frac{V_1(T)}T\to\infty
\end{equation*}
as $T\to\infty$. This completes the proof.

\subsection{Conjectural Bounds}\label{section: conjectures}
In this section we give some evidence in favor of the conjectures stated in the Introduction.
The precise expression for the variance appearing in Proposition~\ref{prop: main} establishes a way to prove even tighter upper bounds, by reducing to combinatorial statements about the polynomials $P_q$, defined in \eqref{eq : polys with correction}.  It is not difficult to see that the vector $(r(t),r'(t)/\si,r''(t)/\si^2)$ always lies in the domain
$$
D = \{(x,y,z) \in \R^3 : x^2 + y^2 \le 1, y^2 + z^2 \le 1 \}.
$$
By Proposition~\ref{prop: divisibility}, $R_q(x,y,z) = P_q(x,y,z)/(x+z)^2$ is a homogeneous polynomial and since $D$ contains all segments to the origin, it follows that $R_q$ attains the maximum of its absolute value on the boundary. We expect that the maximum should be obtained at the points where $|x|=|z|$.

When $x = -z$, the same techniques employed in this paper show the value to be
$$
\frac{P_q(x,y,z)}{(x+z)^2} \Big \vert_{z=-x} = 2^{2q-1} (x^2 + y^2)^{q-1}
$$
and so on this boundary component the value of $R_q$ is $2^{2q-1}$. We believe that this bound is the one relevant to Gaussian processes, however numerical computations suggest that $R_q$ can be much larger at the points where $x=z$. We believe that there is some `hidden' structure that prevents $r(t)$ from being close to $r''(t)/\si^2$ in certain subregions of $D$. For example, if $r(t)$ is close to $1$ then we should be close to a local maximum and so we would expect $r''(t)$ to be negative. Understanding the `true domain' where the vector $(r(t),r'(t)/\si,r''(t)/\si^2)$ `lives' already appears to be a quite interesting question.

\section{Singular spectral measure}
\subsection{Atoms in the spectral measure: proofs of Theorem~\ref{thm: extremal}~\ref{part: quad} and Corollary~\ref{cor: special atom}}
In this section we consider the effect of atoms in the spectral measure, that is, we prove Theorem~\ref{thm: extremal}~\ref{part: quad} and Corollary~\ref{cor: special atom}. Our proof relies on the following proposition.
\begin{prop}\label{prop: quad}
Let $\mu$ be a signed-measure with $\int_\R d|\mu|<\infty$. Then $\mu$ contains an atom if and only if there exists $c>0$ such that
\[
\int_{-T}^T (T-|t|) \, |\widehat{\mu}(t)|^2 dt \ge c T^2
\]
for all $T>0$.
\end{prop}
We postpone the proof of Proposition~\ref{prop: quad} to Section~\ref{sec: quad}. We will also need the following result.
\begin{lem}\label{lem: mean int cosine}
Let $f$ be a SGP with covariance kernel $r$, spectral measure $\rho$ and suppose that $\rho$ has a continuous component. Let $\psi(t) = A \cos(\sigma t + \alpha)$, where $A \in \R$, $\alpha\in[0,2\pi]$ and $\sigma^2 = -r''(0)$. Denote by $N_J(\psi) = \#\{ t \in [0,\pi J/\si] : f(t) = \psi(t) \}$ the number of crossings of the curve $\psi$ by the process. Then $\E [N_J(\psi)] = J$.
\end{lem}

\begin{proof}
Denote the Gaussian density function by $\varphi$ and by $\Phi$ the corresponding distribution function. The generalised Rice formula \cite{CL}*{Equation 13.2.1} gives
\begin{align*}
\E N_J(\psi) &= \sigma \int_{0}^{\frac{\pi J}\si} \varphi(\psi(y)) \left [ 2 \varphi \left( \frac{\psi'(y)}{\sigma} \right) + \frac{\psi'(y)}{\sigma}
 \left( 2 \Phi \left( \frac{\psi'(y)}{\sigma}\right) - 1 \right) \right ]dy \\
 &= \sigma \int_{0}^{\frac{\pi J}\si} \frac{e^{-\frac{A^2}{2} \cos^2 (\sigma y+ \alpha)}}{\sqrt{2 \pi}}
 \left [ \frac{2e^{-\frac{A^2}{2} \sin^2 (\sigma y + \alpha)}}{\sqrt{2 \pi}}  - A \sin (\sigma y+ \alpha)
 \left( 2 \Phi \left( -A \sin (\sigma y+ \alpha) \right) - 1 \right) \right ]dy \\
 &= J e^{-\frac{A^2}{2}}
 -  \frac{\sigma}{\sqrt{2\pi}} \int_{0}^{\frac{\pi J}\si} e^{-\frac{A^2}{2} \cos^2 (\sigma y+ \alpha)} A \sin (\sigma y+ \alpha)  \left( 2 \Phi \left( -A \sin (\sigma y+ \alpha) \right) - 1 \right) dy\\
 &= J e^{-\frac{A^2}{2}}
 -  \frac{\sigma}{\sqrt{2\pi}} \int_{0}^{\frac{\pi J}\si} e^{-\frac{A^2}{2} \cos^2 (\sigma y+ \alpha)} |A| \sin (\sigma y+ \alpha)  \left( 2 \Phi \left( -|A| \sin (\sigma y+ \alpha) \right) - 1 \right) dy.
\end{align*}
Write
\begin{equation*}
F(y) = e^{-\frac{A^2}{2} \cos^2 (y)} |A| \sin (y)  \left( 2 \Phi \left( -|A| \sin (y) \right) - 1 \right)
\end{equation*}
and notice that $F$ is periodic with period $\pi$. This yields
\begin{equation}\label{eq: periodic}
  \E N_J(\psi) = J\left(e^{-\frac{A^2}{2}} -  \frac{\sigma}{\sqrt{2\pi}} \int_{0}^{\frac\pi\si} F(\si y + \alpha) dy \right)= J \left(e^{-\frac{A^2}{2}} -  \frac{1}{\sqrt{2\pi}} \int_{0}^{\pi} F( y ) dy \right).
\end{equation}
Moreover, since $F$ is even we have
\[
\int_{0}^{\pi} F(y) \ dy =
\int_{0}^{\frac{\pi}{2}} F(y) \ dy
+ \int_{\frac{\pi}{2 }}^{\pi} F(y) \ dy =
\int_{0}^{\frac{\pi}{2 }} F(y) \ dy +
\int_{-\frac{\pi}{2 }}^{0} F(y) \ dy =
2 \int_{0}^{\frac{\pi}{2 }} F(y) \ dy.
\]
Substituting $u = |A| \cos(y)$ we obtain
\begin{align*}
-  \frac{1}{\sqrt{2\pi}} \int_{0}^{\pi} F(y) dy & = - {\sqrt{\frac2 \pi}} \int_{0}^{\pi/2} F(y) dy\\
&= -{\sqrt{\frac2 \pi}} \int_{0}^{|A|}  e^{-\frac{u^2}{2}} \cdot \left (2 \Phi \left(-\sqrt{A^2 - u^2} \right) - 1 \right) du \\
&=  \frac{2}{ \pi}  \int_{0}^{|A|}  \int_{0}^{\sqrt{A^2 - u^2}} e^{-\frac{u^2+v^2}{2}} \ dv \ du \\
&=  \frac{2}{ \pi} \int_{0}^{|A|} \int_{0}^{\pi/2} e^{-\frac{r^2}{2}} r d\theta dr
=   1 - e^{-\frac{A^2}{2}}.
\end{align*}
Inserting this value into \eqref{eq: periodic} yields the result.
\end{proof}

\begin{proof}[Proof of Theorem~\ref{thm: extremal}~\ref{part: quad}]
First we note that, by stationarity, $\var[N(T)]\le C T^2$ for some $C>0$. Assume that $\rho$ has an atom at a point different from $\si$. By \eqref{eq: Lb for N} and \eqref{eq: V_1}, to show that $\var[N(T)]\ge \frac{c\si^2}{2\pi^2}T^2$ for some $c>0$ it is enough to see that
\begin{equation*}
  \int_{-T}^T \left(T-|t|\right) \left(r(t) + \frac{r''(t)}{\si^2}\right)^2 dt \ge c T^2.
\end{equation*}
But this follows from Proposition~\ref{prop: quad} if we define the signed measure $\mu$ by $d\mu(\lm) = (1-\frac{\lm^2}{\si^2})d\rho(\lm)$ and notice that $\hat{\mu}=r+\frac{r''}{\si^2}$ and that $\mu$ has an atom.

For the converse, notice that it is enough to check that for integer $J$ we have
\begin{equation*}
  \frac{\var[N(\frac\pi\si J)]}{J^2}\to 0 \quad\text{ as }\quad J\to\infty,
\end{equation*}
since this implies that $\var[N(T)]=o(T^2)$, by stationarity. Assume first that $\rho$ has no atoms; we adapt the proof of \cite{BF}*{Thm 4}.
By the Fomin-Grenander-Maruyama theorem, $f$ is an ergodic process (see, e.g., \cite{Gre}*{Sec. 5.10}). By standard arguments, this also implies that the sequence
\begin{equation*}
  \cN_j=\# \left\{ t\in \Big[(j-1)\frac\pi\si,j\frac\pi\si\Big): \: f(t)=0\right\}.
\end{equation*}
is ergodic. Recall that we assume the Geman condition, which implies that the first and second moments of
\begin{equation*}
  N\left(\frac\pi\si J\right)=\sum_{j=1}^J \cN_j
\end{equation*}
are finite. Thus, by von Neumann's ergodic theorem, we have
\begin{equation*}
\lim_{J\to\infty}\frac {N(\frac\pi\si J)} J  = \E[\cN_1]=1,
\end{equation*}
where the convergence is both in $L^1$ and $L^2$ (see~\cite{Walters}*{Cor. 1.14.1}). We conclude that
\begin{equation*}
  \lim_{J\to\infty}\frac{\var[N(\frac\pi\si J)]}{J^2}  =0.
\end{equation*}

Finally suppose that $\rho = \theta \rho_c + (1-\theta) \delta^*_{\si}$ where $0<\theta<1$ and $\rho_c$ has no atoms. We may represent the corresponding process as
\begin{equation*}
  f(t)=\sqrt{\theta}f_c + \sqrt{(1-\theta)X}\cos(\si t + \Phi)
\end{equation*}
where $f_c$ is a SGP with spectral measure $\rho_c$, $X\sim \chi^2(2)$, $\Phi\sim\text{Unif}([0,2\pi])$, and moreover $f_c,X$ and $\Phi$ are pairwise independent. By the law of total variance and Lemma~\ref{lem: mean int cosine} we have
\begin{align}
\nonumber
  \var\Big[N\Big(\frac\pi\si J\Big)\Big]&=\E\Big[\var\big[N\Big(\frac\pi\si J\Big)\Big|X,\Phi\big]\Big]+\var\Big[\E\big[N\Big(\frac\pi\si J\Big)\Big|X,\Phi\big]\Big] \\
  &= \E\Big[\var\big[N\Big(\frac\pi\si J\Big)\Big|X,\Phi\big]\Big].
  \label{eq : total variance}
\end{align}
We define, for $A\in\R$ and $\alpha\in[0,2\pi]$,
\begin{equation*}
  \cN^{A,\alpha}_j=\# \left\{ t\in \Big[(j-1)\frac\pi\si,j\frac\pi\si\Big): \: f_c(t)=A\cos(\si t+\alpha)\right\}.
\end{equation*}
As before the process $f_c$ is ergodic, and so is the sequence $\cN^{A,\alpha}_j$ for fixed $A$ and $\alpha$. This implies that
\begin{equation*}
  \lim_{J\to\infty}\frac{\var[N(\frac\pi\si J)|X,\Phi]}{J^2}  =0
\end{equation*}
(almost surely), exactly as before. Furthermore, using stationarity we have
\begin{equation*}
  \frac{1}{J^2}\var\big[N\Big(\frac\pi\si J\Big)\Big|X,\Phi\big]\le \var\big[N\Big(\frac\pi\si\Big)\Big|X,\Phi\big]
\end{equation*}
and using \eqref{eq : total variance} we see that
\begin{equation*}
  \E\Big[\var\big[N\Big(\frac\pi\si\Big)\Big|X,\Phi\big]\Big]=\var\big[N\Big(\frac\pi\si\Big)\big]<+\infty,
\end{equation*}
since we assume the Geman condition. It follows from dominated convergence that
\begin{equation*}
  \lim_{J\to\infty}\frac{1}{J^2} \E\Big[\var\big[N\Big(\frac\pi\si J\Big)\Big|X,\Phi\big]\Big] =0
\end{equation*}
whence $ \lim_{J\to\infty}\frac{\var[N(\frac\pi\si J)]}{J^2}  =0$.
\end{proof}

\begin{proof}[Proof of Corollary~\ref{cor: special atom}]
Let $M = \limsup_{|t| \rightarrow \infty} \varphi(t)$, where $\varphi$ is defined in \eqref{eq: phi}. By assumption we have $M < 1$ and we define
\[
\theta_0 = \frac{1 - M}{\sqrt2 - M}.
\]
We would like to apply Theorem~\ref{thm: UB} to the spectral measure $\rho_\theta$. Writing $r_\theta = \cF[\rho_\theta]$ and $r = \cF[\rho]$ we have
$r_\theta(t) = (1-\theta) r(t)  + \theta \cos(\si t)$, and $\si_\theta^2 = -r_\theta''(0) = \si^2$.
We accordingly compute
\begin{align*}
  \p_{\theta}(t) &= \max \left\{ |r_\theta(t)| +\frac{|r'_\theta(t)|}{\sigma}, \, \frac{|r''_\theta(t)|}{\sigma^2} + \frac{|r'_\theta(t)|}{\sigma} \right\} \\
  &\le \max \Big\{ (1-\theta)\left(|r(t)| +\frac{|r'(t)|}{\sigma}\right) + \theta (|\cos \si t|+|\sin \si t|), \\
  & \qquad \qquad \qquad(1-\theta)\left(\frac{|r''(t)|}{\sigma^2} + \frac{|r'(t)|}{\sigma}\right) + \theta (|\cos \si t|+|\sin \si t|)\Big\} \\
  &\le (1-\theta)M + \theta \sqrt2
\end{align*}
and so
\[
\limsup_{|t| \rightarrow \infty} \varphi_\theta(t) <  1
\]
for $\theta<\theta_0$. Applying Theorem~\ref{thm: UB} to $\rho_\theta$ and to $\rho$ we obtain
\begin{align*}
\var[N(\rho_\theta; T) ] & \asymp  T\int_{-T}^T \left(1-\frac{|t|}{T}\right) \left(r_\theta(t)+\frac{r_\theta''(t)}{\si_\theta^2}\right)^2 \ dt
\\ &= (1-\theta)^2
T\int_{-T}^T \left(1-\frac{|t|}{T}\right) \left(r(t)+\frac{r''(t)}{\si^2}\right)^2 \ dt \\
& \asymp \var[N(\rho;T)]. \qedhere
\end{align*}

\end{proof}

\subsection{Proof of Proposition~\ref{prop: quad}}\label{sec: quad}
We begin with a review of some elementary harmonic analysis that we will need, for more details and proofs see, e.g., Katznelson's book \cite{Katz}*{Ch. VI}. Let $\M$ denote the space of all finite signed measures on $\R$ endowed with the \emph{total mass} norm $\|\mu\|_1 = \int_\R d |\mu|$. Recall that the \emph{convolution} of two measures $\mu, \nu \in \M$ is given by $(\mu*\nu) (E) = \int \mu(E-\lm) d\nu(\lm)$ for any measurable set $E$ and satisfies $\|\mu *\nu\|_1 \le \|\mu\|_1 \|\nu\|_1$ and $\cF[\mu * \nu] = \cF[\mu]\cdot \cF[\nu]$.
Moreover, $\cF[\cdot]$ is a uniformly continuous map with $\|\cF[\mu]\|_\infty \le \|\mu\|_1$. We identify a function $f\in L^1$ with the measure whose density is $f$.

The following lemma is a version of Parseval's identity, see \cite{Katz}*{VI 2.2}.
\begin{lem}[Parseval]\label{lem: parseval}
If $f,\cF[f]\in L^1(\R)$ and $\nu\in\M$, then $\int f d\nu =\frac{1}{2\pi}\int \cF[f] \overline{\cF[\nu]}$.
\end{lem}
A simple application of Parseval's identity proves our next lemma.
\begin{lem}\label{lem: parseval with kernel}
Suppose that $\mu,\nu\in \M$ and $S, \cF[S]\in L^1(\R)$. Then
\[
\int (S *\mu) d\nu = \frac 1 {2\pi} \int \cF[S] \cF[\mu] \overline{\cF[\nu]}.
\]
\end{lem}

\begin{proof}
Note that $S * \mu$ is a function and further that
\begin{align*}
 \|S *\mu\|_1
 & \le \| \mu \|_1 \| S\|_1 <\infty, \quad \text{ and}
\\ \| \cF[S*\mu] \|_1 & = \| \cF[S] \cF[\mu] \|_1 \le \|\cF[\mu]\|_\infty \|\cF[S]\|_1 \le \|\mu\|_1 \|\cF[S]\|_1 <\infty.
\end{align*}
A simple application of Lemma~\ref{lem: parseval} finishes the proof.
\end{proof}

We will also use the so-called `triangle function'
\[
\mathcal{T}_T(t) = \left(1-\frac{|t|}{T} \right)\ind_{[-T,T]}(t)
\]
which satisfies $\mathcal{T}_T = \cF[\mathcal{S}_T]$ where\footnote{We use the normalisation $\sinc(x)=\frac{\sin x}{x}$.}
\begin{equation*}
  \mathcal{S}_T(\lm) = \frac{T}{2\pi} \  \sinc^2\left(\frac{T \lm}{2}\right).
\end{equation*}
Notice that applying Lemma~\ref{lem: parseval with kernel} to these functions, we obtain
\begin{equation*}
  \int_{-T}^T \left(1-\frac{|t|}{T} \right) |\widehat{\mu}(t)|^2 dt =  \int_\R \mathcal{T}_T |\cF[\mu]|^2 =2\pi  \int (\mathcal{S}_T * \mu)\ d\mu,
\end{equation*}
which is \eqref{eq: parseval for key int}.

We are now ready to prove Proposition~\ref{prop: quad}. First suppose that $\mu$ contains an atom at $\al$. Write $\mu = \mu_1+\mu_2$ where $\mu_1 = c\delta_\al$ for some $c\ne 0$ and $\mu_2(\{\al\})=0$. Note that
\begin{equation}\label{eq: mu2 near al}
|\mu_2([\al-\ep,\al+\ep])|\le |\mu_2|([\al-\ep,\al+\ep])\downarrow 0, \text{   as  }\ep\downarrow 0.
\end{equation}
We have
\begin{align*}
| \cF[\mu](t)|^2 &= |\cF[\mu_1](t)|^2 + 2\text{Re} \{\cF[\mu_1](t)\overline{\cF[\mu_2](t)}\} + |\cF[\mu_2](t)|^2
\\ & \ge |c|^2 + 2\text{Re}\{ \cF[\mu_1](t) \overline{\cF[\mu_2](t)}\}
\end{align*}
Using this and Lemma~\ref{lem: parseval with kernel} we obtain
\begin{align*}
\int_{-T}^{T} (T-|t|) |\widehat{\mu}(t)|^2 dt = T \int_\R \mathcal{T}_T |\cF[\mu]|^2
& \ge |c|^2 T \int_\R \mathcal{T}_T  + 2T \text{Re}\left\{ \int_\R \mathcal{T}_T \cF[\mu_1] \overline{\cF[\mu_2]} \right\}
\\ & = |c|^2 T^2 + 4\pi T \, \text{Re} \left\{\int_\R \mathcal{S}_T * \mu_1 \ d\mu_2\right\}.
\end{align*}
It is therefore enough to show that $\int_\R(\mathcal{S}_T*\mu_1)\ d\mu_2 = o(T)$. We bound
\[
\left|\int (\mathcal{S}_T*\mu_1)(\lm) \ d\mu_2(\lm)\right|
= \left|\frac{cT}{2\pi}\int_\R \sinc^2\big(\tfrac{T}{2}(\lm-\al)\big)  \ d\mu_2(\lm)\right|
\le \frac{|c|T}{2\pi} \int_\R\sinc^2\big(\tfrac{T}{2}(\lm-\al)\big)  \ d|\mu_2|(\lm).
\]

Let $I_\al(T) = \big[\al - \frac{\log T}{T},\al+\frac{\log T}{T}\big]$. By \eqref{eq: mu2 near al} we have
\[
\int_{I_\al(T)} \sinc^2 \big(\tfrac {T}{2}(\lm-\al)\big) d|\mu_2|(\lm) \le |\mu_2|(I_\al(T))\to 0, \quad \text{as  } T\to \infty.
\]
On $\R \setminus I_\al(T)$ we have $\frac{T}{2}|\lm-\al|\ge \frac{\log T}{2}$, so that
\[
\int_{\R \setminus I_\al(T)} \sinc^2 \big(\tfrac {T}{2}(\lm-\al)\big) d|\mu_2|(\lm)
\le \frac 4 {(\log T)^2} |\mu_2|(\R) \to 0, \quad \text{as  } T\to\infty.
\]
This concludes the first part of the proof.

\medskip
Conversely, suppose that $\mu$ contains no atoms. Recall that
\begin{equation*}
  \int_{-T}^T \left(1-\frac{|t|}{T} \right) |\widehat{\mu}(t)|^2 dt = 2\pi  \int (\mathcal{S}_T * \mu)\ d\mu.
\end{equation*}
We will show that $|(\mathcal{S}_T*\mu)(\lm)|= o(T)$, uniformly in $\lm$, which will conclude the proof.
As before, denoting $I_\lm(T) = \big[\lm - \frac{\log T}{T},\lm+\frac{\log T}{T}\big]$ we have
\begin{equation*}
|( \mathcal{S}_T*\mu )(\lm)|
 = \left|\int_\R \frac{T}{2\pi} \sinc^2\big(\tfrac{T}{2}(\lm-\tau \big) d\mu(\tau) \right|
 \le \frac{T}{2\pi} \left( |\mu|(I_\lm(T)) + \frac{4 |\mu|(\R)}{ (\log T)^2} \right).
\end{equation*}
It therefore suffices to prove the following claim.

\begin{clm}\label{clm: eugene}
Let $\nu$ be a non-negative, finite measure on $\R$ that contains no atoms. Then
\[
\sup_{x\in\R} \nu\big([x-\ep,x+\ep]\big) \to 0, \quad \text{as  } \ep\downarrow 0.
\]
\end{clm}

\begin{proof}
Denote  $B(x, \ep) =[x-\ep,x+\ep]$ and
$m(\ep) = \sup_{x\in\R} \nu\big(B(x, \ep)\big)$. It is clear that $m(\ep)$ decreases with $\ep$ so $m(\ep)$ must converge as $\ep \downarrow 0$ to some non-negative limit, $2\delta\ge 0$. Suppose that $\delta>0$ and choose $N>0$ such that
$\nu(\R\setminus [-N/2,N/2])<\delta$.
Fix $n\in \N$ and divide $[-N,N]$ into disjoint `dyadic' intervals
\[ D_n  = \left\{ [kN2^{-n},(k+1) N2^{-n} ):   k\in \Z \cap [-2^n, 2^n) \right\}.  \]

For any $x\in \R$, either $B(x,\frac N{2^n})\subseteq \R\setminus [-N/2,N/2]$, which implies that $\nu(B(x,\frac N{2^n}))<\delta$,
or $B(x,\frac N{2^n})\subseteq I \cup I'$ for some $I, I' \in D_{n-1}$. Therefore,
\[
m\big(\tfrac{N}{2^n}\big) \le \max\left(\delta, 2\sup_{I\in D_{n-1}} \nu(I) \right).
\]
Recall that by definition of $\delta$ we have $m\big(\tfrac{N}{2^n}\big)\ge 2\delta$. We conclude that for every $n \in \N$ we can find $I_n \in D_n$ such that
\begin{equation}\label{eq: I>delta}
\nu(I_{n}) \ge \delta.
\end{equation}

Next we shall construct a sequence of nested dyadic intervals $\{J_n\}_{n=0}^\infty$ such that, for all $n$,
\begin{equation*}
J_n \in D_n, \quad J_{n+1}\subseteq J_{n}, \quad  \nu(J_n) \ge \delta.
\end{equation*}
This will imply, by Cantor's lemma, that $\bigcap_{n} J_n = \{x\}$ for some $x\in\R$, and further that
$\nu(\{x\}) = \lim_{n\to\infty}\nu(J_n) \ge \delta >0$. This contradicts the assumption that $\nu$ has no atoms, which will end our proof.

We start by setting $J_0 = [-N,N]$.
Suppose that we have constructed $J_0\supset J_1 \supset J_2 \supset \dots\supset J_m$ such that for every $n>m$ we can find $I_n' \in D_n$ that satisfies
\begin{equation}\label{eq: J2}
I_n'\subset J_m, \quad\text{and}\quad \nu(I_n') \ge \delta;
\end{equation}
that is, the interval $J_m$ has a descendant of any generation whose $\nu$-measure is at least $\delta$. Notice that this holds for $m=0$ by \eqref{eq: I>delta}. Notice that if \eqref{eq: J2} fails for both descendants of $J_m$ in the generation $D_{m+1}$, then it also fails for $J_m$, since $\nu(J)\ge\nu(J')$ for every descendant $J'\subseteq J$. This completes the inductive construction of $J_m$ and consequently the proof.
\end{proof}

\subsection{Singular continuous measures: Proof of Proposition~\ref{prop: sing var}}\label{sec: cont sing}
Throughout this section we assume the notation of Section~\ref{subsec: sing cnts exa} and that $\alpha_n>R_n$ for every $n\ge 1$. We begin with the following observation.
\begin{lem} \label{lem: cont sing measure bound}
  (i) Suppose that $I$ is an interval and $|I|< R_n$. Then $\rho(I)\le2^{-n}$.\\
  (ii) Let $\delta>0$ and suppose that $R_n<\frac\delta4$. Then $\rho((\lm-\delta,\lm+\delta))\ge 2^{-n}$ for any $\lm$ in the support of $\rho$.
\end{lem}
\begin{rmks} \phantom{w}
  \begin{enumerate}[label=\arabic*.]
    \item A probability measure $\mu$ is said to be \emph{exact dimensional} if there exists $d$ such that
        \begin{equation}\label{eq: exact dim}
          \lim_{\delta\to 0}\frac{\log \mu((\lm-\delta,\lm+\delta))}{\log\delta}=d
        \end{equation}
        for $\mu$-almost every $\lm$. It follows from this lemma that the measure $\rho_a$ is exact dimensional for $0<a<\frac12$, with $d=\log2/\log\frac1a$, and moreover
        \begin{equation*}
          \left|\frac{\log \rho_a((\lm-\delta,\lm+\delta))}{\log\delta}-d\right|= O\left( \frac 1{\log \frac1\delta}\right)
        \end{equation*}
        uniformly for every $\lm$ in the support of $\rho_a$.

    \item If $\frac12\le a <1$ then the measure $\rho_a$ is also exact dimensional, but understanding more detailed properties seems to be a difficult question, related to certain notions of entropy. We refer the reader to the surveys \cites{G,V} for a more thorough discussion.

    \item For any compactly supported exact dimensional spectral measure $\rho$, of dimension $d$, such that \eqref{eq: exact dim} converges uniformly for $\rho$-almost every $\lm$, we may imitate the proof of Proposition~\ref{prop: sing var} to yield
        \begin{equation*}
          cT^{2-d-\ep}\le T\int_{0}^T \left(1-\frac{t}{T}\right) \left(r(t)+\frac{r''(t)}{\si^2}\right)^2 \ dt  \le C T^{2-d+\ep}
        \end{equation*}
        for any $\ep>0$ and  some $c,C>0$.
  \end{enumerate}
\end{rmks}
\begin{proof}
  (i) We begin by recalling Kershner and Wintner's proof that the spectrum is a Cantor type set. Set $C_0=[-R_0,R_0]$. We form $C_1$ by deleting an interval of length $2(\alpha_1-R_1)$ from the centre of $C_0$, which yields $C_1=[-R_0,-\alpha_1+R_1]\cup[\alpha_1-R_1,R_0]$. We inductively construct $C_n$ to consist of $2^n$ (closed) intervals formed by deleting an interval of length $2(\alpha_n-R_n)$ from the centre of each of the intervals in the previous generation. Then $S=\cap_{n=0}^\infty C_n$ is the spectrum.

  Let $\rho_n=\delta_{\alpha_1}^* * \delta_{\alpha_2}^* *\dots* \delta_{\alpha_n}^*$ which is supported on the set
  \begin{equation*}
    S_n=\left\{\sum_{j=1}^n \ep_j \alpha_j : \ep_j=\pm 1\right\}
  \end{equation*}
  and assigns a mass of $2^{-n}$ to each of the points of $S_n$. Furthermore, the elements of $S_n$ are the midpoints of the intervals that make up $C_n$, and if $s,s'\in S_n$ are distinct then $|s-s'|\ge2\alpha_n>2R_n$. We conclude that if $|I|< R_n$ then $I$ contains at most one element of $S_n$ and so $\rho_n(I)\le 2^{-n}$. Moreover $I$ can only intersect with the descendants of one of the intervals of $C_n$, and so $\rho_m(I)\le 2^{-n}$ for every $m\ge n$.

  It remains to note that $\rho_m$ converges weakly to $\rho$, and that an interval is always a continuity set of the measure $\rho$, whence
  \begin{equation*}
    \rho(I)=\lim_{m\to\infty}\rho_m(I)\le 2^{-n}.
  \end{equation*}

  \medskip\noindent(ii) In fact the measure $\rho_n$ introduced above is the law of the random variable
  \begin{equation*}
    X_n=\sum_{j=1}^n \ep_j \alpha_j
  \end{equation*}
  where $\ep_j$ denotes a sequence of i.i.d. Rademacher\footnote{That is, $\Pro[\ep_j=1]=\Pro[\ep_j=-1]=\frac12$.} random variables, while $\rho$ is the law of the infinite sum
  \begin{equation}\label{def: X}
    X=\sum_{j=1}^\infty \ep_j \alpha_j.
  \end{equation}
  For $\lm$ in the support of $\rho$ we can write $\lm=\sum_{j=1}^\infty \widetilde{\ep}_j \alpha_j$ for some fixed sequence $\widetilde{\ep}_j\in\{-1,1\}$ and we write $\lm_n=\sum_{j=1}^n \widetilde{\ep}_j \alpha_j$.

  Notice that $|X-\lm|\le|X_n-\lm_n|+2R_n$. If $R_n<\frac\delta4$ we therefore have
  \begin{equation*}
    \Pro\big[|X-\lm|<\delta\big]\ge\Pro\big[|X_n-\lm_n|<2R_n\big]=\Pro[\ep_j=\widetilde{\ep}_j\text{ for }1\le j\le n]=2^{-n}.\qedhere
  \end{equation*}
\end{proof}

We now proceed to prove Proposition~\ref{prop: sing var}.

\begin{proof}[Proof of Proposition~\ref{prop: sing var}]
  Using \eqref{eq: parseval for key int} we see that we wish to prove that
  \begin{equation}\label{eq: goal sing prop}
    \iint_{\R^2} \sinc^2\left(\frac T2(\lm-\lm')\right)d\mu(\lm)d\mu(\lm')\asymp 2^{-M}
  \end{equation}
  where, as before, $d\mu(\lm)=\left(1-\frac{\lm^2}{\si^2}\right)d\rho(\lm)$. Fix $A>4$ (to be chosen large) and $\frac12<\beta<1$ and define $\widetilde{M}$ by $R_{\widetilde{M}}<\frac{2A}T\le R_{\widetilde{M}-1}$. We claim that\footnote{Throughout this proof $c,c',C$ and $C'$ denote positive constants whose exact value is irrelevant and which may vary from one occurrence to the next. They may depend on the measure $\rho$ but are independent of $A$ and $T$.}
  \begin{equation*}
    \eta2^{-M}-\frac {CA}T 2^{-\widetilde{M}}\le \iint\displaylimits_{|\lm-\lm'|<\frac AT} \sinc^2\left(\frac T2(\lm-\lm')\right)d\mu(\lm)d\mu(\lm') \le C' 2^{-\widetilde{M}}
  \end{equation*}
  for some constant $\eta>0$ (depending only on $\rho$) and that
  \begin{equation*}
    \iint\displaylimits_{\frac {2A}T\le|\lm-\lm'|<2^{\beta M}\frac {2A}T} \sinc^2\left(\frac T2(\lm-\lm')\right)d|\mu|(\lm)d|\mu|(\lm')\le \frac{C2^{-\widetilde{M}}}{A^2}.
  \end{equation*}

  Let us first see how the claims imply the proposition. Since $R_n<\frac12R_{n-1}$ we see that
  \begin{equation*}
    \frac{1}T\le R_{M-1}\le 2^{\widetilde{M}-M-1}R_{\widetilde{M}}<2^{\widetilde{M}-M}\frac AT
  \end{equation*}
  which yields $2^{M-\widetilde{M}}<A$. Inserting this into our claims and combining them we get
  \begin{equation*}
    \left(\eta-C\Big(\frac1A+\frac {A^2}T \Big)\right)2^{-M}\le \iint\displaylimits_{|\lm-\lm'|<2^{\beta M}\frac {2A}T} \sinc^2\left(\frac T2(\lm-\lm')\right)d\mu(\lm)d\mu(\lm') \le C' \Big(A+\frac1A\Big)2^{-M}.
  \end{equation*}
  We now fix $A$ so large that the lower bound in this expression is at least $\left(\frac\eta2-\frac {CA^2}T \right)2^{-M}$. We then have
  \begin{equation*}
    \frac\eta4 2^{-M}\le \iint\displaylimits_{|\lm-\lm'|<2^{\beta M}\frac {2A}T} \sinc^2\left(\frac T2(\lm-\lm')\right)d\mu(\lm)d\mu(\lm') \le C' \Big(A+\frac1A\Big)2^{-M}
  \end{equation*}
  for sufficiently large $T$.

  We finally bound
  \begin{equation*}
    \iint\displaylimits_{|\lm-\lm'|\ge2^{\beta M}\frac {2A}T} \sinc^2\left(\frac T2(\lm-\lm')\right)d|\mu|(\lm)d|\mu|(\lm')\le \frac{2^{-2\beta M}} {A^2} \iint\displaylimits_{\R^2} d|\mu|(\lm)d|\mu|(\lm') \le \frac C{A^2} 2^{-2\beta M} =o(2^{-M}),
  \end{equation*}
  since $\beta>\frac12$ and $A$ has been fixed. This yields \eqref{eq: goal sing prop}.

  It remains only to establish the two claims. We begin with the terms near the diagonal. Since the spectrum is compact we note that $1-\frac{\lm^2}{\si^2}$ is bounded on the support of $\rho$, and using part (i) of Lemma~\ref{lem: cont sing measure bound} we see that
  \begin{align*}
    \iint\displaylimits_{|\lm-\lm'|<\frac AT} \sinc^2\left(\frac T2(\lm-\lm')\right)d\mu(\lm)d\mu(\lm')& \le C \iint\displaylimits_{|\lm-\lm'|<\frac AT} d\rho(\lm)d\rho(\lm')\\
    &= C \int_\R \rho\left(\left(\lm-\frac AT,\lm+\frac AT\right)\right)d\rho(\lm)\le C2^{-\widetilde{M}}.
  \end{align*}
  In the other direction we note that
  \begin{align*}
    \iint\displaylimits_{|\lm-\lm'|<\frac AT} \sinc^2\left(\frac T2(\lm-\lm')\right)d\mu(\lm)d\mu(\lm')& = \iint\displaylimits_{|\lm-\lm'|<\frac AT} \sinc^2\left(\frac T2(\lm-\lm')\right) \left(1-\frac{\lm^2}{\si^2}\right)^2 d\rho(\lm)d\rho(\lm')\\
    &\qquad + \iint\displaylimits_{|\lm-\lm'|<\frac AT}\sinc^2\left(\frac T2(\lm-\lm')\right) \frac{\lm^2-(\lm')^2}{\si^2} d\mu(\lm)d\rho(\lm').
  \end{align*}
  The integrand of the first term is positive and so it is bounded from below by
  \begin{equation*}
    \iint\displaylimits_{\substack{|\lm-\lm'|<\frac 4T \\ ||\lm|-\si|>c}}\sinc^2\left(\frac T2(\lm-\lm')\right) \left(1-\frac{\lm^2}{\si^2}\right)^2 d\rho(\lm)d\rho(\lm')\ge c' \iint\displaylimits_{\substack{|\lm-\lm'|<\frac 4T \\ ||\lm|-\si|>c}} d\rho(\lm)d\rho(\lm')\ge c' 2^{-M}
  \end{equation*}
  where we have used part (ii) of Lemma~\ref{lem: cont sing measure bound}. We also estimate, similar to before,
  \begin{equation*}
    \left|\iint_{|\lm-\lm'|<\frac AT}\sinc^2\left(\frac T2(\lm-\lm')\right) \frac{\lm^2-(\lm')^2}{\si^2} d\mu(\lm)d\rho(\lm')\right|\le C\frac AT \iint\displaylimits_{|\lm-\lm'|<\frac AT}d\rho(\lm)d\rho(\lm')\le C\frac AT 2^{-\widetilde{M}}
  \end{equation*}
  and we have therefore shown that
  \begin{equation*}
    \eta2^{-M}-\frac {CA}T 2^{-\widetilde{M}}\le \iint\displaylimits_{|\lm-\lm'|<\frac AT} \sinc^2\left(\frac T2(\lm-\lm')\right)d\mu(\lm)d\mu(\lm') \le C' 2^{-\widetilde{M}}
  \end{equation*}
  where $\eta>0$, as desired.

  We finally estimate the off-diagonal contribution. We have
  \begin{equation*}
    \iint\displaylimits_{2^j\frac {A}T\le|\lm-\lm'|<2^{j+1}\frac {A}T} \sinc^2\left(\frac T2(\lm-\lm')\right)d|\mu|(\lm)d|\mu|(\lm')\le \frac{C}{2^{2j}A^2}\iint\displaylimits_{2^j\frac {A}T\le|\lm-\lm'|<2^{j+1}\frac {A}T} d\rho(\lm)d\rho(\lm').
  \end{equation*}
  For $0\le j \le \beta M$ we notice that $\frac{2^{\beta M}}T<2^{\beta M}R_M=o(2^{M}R_M)=o(1)$ and, since $R_n<\frac12R_{n-1}$, we see that $R_{\widetilde{M}-j}\ge2^{j-1}R_{\widetilde{M}-1}\ge\frac {2^jA}T$. Applying part (i) of Lemma~\ref{lem: cont sing measure bound} once more we get
  \begin{equation*}
    \iint\displaylimits_{2^j\frac {A}T\le|\lm-\lm'|<2^{j+1}\frac {A}T} d\rho(\lm)d\rho(\lm')\le 2^{j-\widetilde{M}}.
  \end{equation*}
  Summing we get
  \begin{equation*}
    \iint\displaylimits_{\frac {2A}T\le|\lm-\lm'|<2^{\beta M}\frac {2A}T} \sinc^2\left(\frac T2(\lm-\lm')\right)d|\mu|(\lm)d|\mu|(\lm')\le \frac{C2^{-\widetilde{M}}}{A^2}\sum 2^{-j} = \frac{C2^{-\widetilde{M}}}{A^2}
  \end{equation*}
  as claimed.
\end{proof}

\section{Proof of Proposition~\ref{prop: divisibility} }\label{section: proof of divisibility}
\subsection{Dehomogenisation}

Our first step is based on the following lemma.
\begin{lem}\label{lem: homo}
Let $P(x,y,z)$ be a homogeneous polynomial. Then $(x+z)^{2}\mid P(x,y,z)$ if and only if $P(-1,y,1)=0$ and $\frac{\partial P}{\partial x}(-1,y,1)=0$.
\end{lem}

\begin{proof}
Consider the polynomial $f(x,y)=P(x,y,1)$ and write $f$ as a polynomial in $x+1$ to obtain $f(x,y)=\sum_{j=0}^{d}a_{j}(y)\cdot(x+1)^{j}$.
Suppose first that
\begin{equation}
a_{0}(y)=f(-1,y)=P(-1,y,1)=0\label{eq:a_0}
\end{equation}
and
\begin{equation}
a_{1}(y)=\frac{\partial f}{\partial x}(-1,y)=\frac{\partial P}{\partial x}(-1,y,1)=0.\label{eq:a_1}
\end{equation}
It follows that $(x+1)^{2}\mid f(x,y)$, and we write $f(x,y)=(x+1)^{2}g(x,y)$.

As $P(x,y,z)$ is homogeneous, one has
\begin{align*}
  P(x,y,z)&=z^{\deg P}P\left(\frac{x}{z},\frac{y}{z},1\right)=z^{\deg P}f\left(\frac{x}{z},\frac{y}{z}\right)\\
  &=z^{\deg P}\left(\frac{x}{z}+1\right)^{2}g\left(\frac{x}{z},\frac{y}{z}\right)=(x+z)^{2}\cdot z^{\deg P-2}g\left(\frac{x}{z},\frac{y}{z}\right).
\end{align*}
Finally $z^{\deg P-2}g\left(\frac{x}{z},\frac{y}{z}\right)$
is a homogeneous polynomial, and we are done.

For the converse, note that if $(x+z)^{2}\mid P(x,y,z)$, then $(x+1)^{2}\mid f(x,y)$,
hence equations \eqref{eq:a_0} and \eqref{eq:a_1} hold.
\end{proof}

In light of Lemma~\ref{lem: homo}, Proposition~\ref{prop: divisibility} is equivalent to the next proposition.
\begin{prop} \label{prop: P_q vanish} For all $q \ge 1$ we have
\begin{enumerate}[label={\rm (\alph*)}]
\item $P_{q}(-1,y,1)=0$, and
\item $\frac{\partial P_{q}}{\partial x}(-1,y,1)=0$.
\end{enumerate}
\end{prop}
We shall therefore concentrate on proving Proposition~\ref{prop: P_q vanish}.

\subsection{Reduction to a combinatorial identity}

For $z\in\mathbb{R}$ and $k\in\mathbb{Z}$, we use the standard notation $(z)_{k}$ for the rising factorial Pochhammer symbol
$$
(z)_{k}=z(z+1)\cdot\cdots\cdot(z+k-1)=\frac{\Gamma(z+k)}{\Gamma(z)}
$$
where the second equality holds for $z$ not a non-positive integer.
We next reformulate Proposition~\ref{prop: P_q vanish} in terms of the purely hypergeometric terms
\begin{equation*}
H_q(l_{1},l_{2},k)=\frac{(-1)^{l_1+l_2}\left(-\frac{1}{2}\right)_{l_1} \cdot \left(-\frac{1}{2}\right)_{l_2}\cdot \left(\frac{1}{2}\right)_{q-l_{1}}\cdot \left(\frac{1}{2}\right)_{q-l_{2}}}{(2q-l_{1}-l_{2}-k)!(l_{2}-l_{1}+k)!(l_{1}-l_{2}+k)!(l_{1}+l_{2}-k)!}
\end{equation*}
and
\begin{equation*}
H'_q(l_{1},l_{2},k)= (2q-l_1-l_2-k) H_q(l_1, l_2, k),
\end{equation*}
in order to be able to apply Zeilberger's algorithm in Section~\ref{subsubsec: proof of first identity}.
We note that $H_q, H'_q$ are defined for every $k, l_1, l_2 \in \Z$, by expressing everything in terms of the Gamma function.

\begin{prop} \label{prop:For-all-,sum(H_k)} For all $q\ge1$ we have
\begin{enumerate}[label={\rm (\alph*)}]
\item \label{prop: sum H_k a}
\begin{equation*}
  \sum_{l_{1},l_{2}}H_q(l_{1},l_{2},k)=\begin{cases}
0, & \text{for }k\ge 2,\\
2^{-4q}(2q-1)c_q, & \text{for }k=1,\\
2^{-4q}c_q, & \text{for }k=0.
\end{cases}
\end{equation*}
\item \label{prop: sum H_k b}
\begin{equation*}
\sum_{l_{1},l_{2}}H'_q(l_{1},l_{2},k)=\begin{cases}
0, & \text{for }k\ge 2,\\
2^{-4q}(2q-1)(2q-2)c_q, & \text{for }k=1,\\
2^{-4q}(2q-1)c_q , & \text{for }k=0.
\end{cases}
\end{equation*}
\end{enumerate}
\end{prop}

\begin{proof} [Proof that Proposition~\ref{prop:For-all-,sum(H_k)} is equivalent to Proposition~\ref{prop: P_q vanish}]
 A rearrangement of the terms in \eqref{eq : polys with correction} yields
\begin{equation}\label{eq: Pq and Sq rel}
P_{q}(-1,y,1)=\sum_{k=0}^{q}(-1)^kd_{q}(k)\cdot y^{2k}+c_{q}\left((2q-1)y^{2}-1\right),
\end{equation}
where
\begin{equation*}
d_{q}(k)=\sum_{\substack{k\le l_{1}+l_{2}\le2q-k\\|l_{1}-l_{2}|\le k}}a_{q}(l_{1})a_q(l_{2})b_{q}(l_{1},l_{2},l_{1}+l_{2}-k)\cdot(-1)^{l_{1}+l_{2}}.
\end{equation*}
Similarly, one obtains
\begin{equation*}
\frac{\partial P_{q}}{\partial x}(-1,y,1)=\sum_{k=0}^{q}(-1)^kd'_{q}(k)\cdot y^{2k}+c_{q}(2q-1)\left(1-(2q-2)y^{2}\right)
\end{equation*}
where
\begin{equation*}
d'_{q}(k)=\sum_{\substack{k\le l_{1}+l_{2}\le2q-k\\|l_{1}-l_{2}|\le k}}(2q-l_{1}-l_{2}-k)\cdot a_{q}(l_{1})a_q(l_{2})b_{q}(l_{1},l_{2},l_{1}+l_{2}-k)\cdot(-1)^{l_{1}+l_{2}-1}.
\end{equation*}
It is therefore enough to prove that
\begin{equation*}
(-1)^{l_1+l_2}a_{q}(l_{1})a_q(l_{2})b_{q}(l_{1},l_{2},l_{1}+l_{2}-k) = 2^{4q} H_q(l_1, l_2, k),
\end{equation*}
which is easily verified by standard algebraic manipulations.
\end{proof}

\subsection{Proof of Proposition~\ref{prop:For-all-,sum(H_k)}~\ref{prop: sum H_k a}}
\label{subsubsec: proof of first identity}

We will use the multivariate
Zeilberger algorithm for multi-sum recurrences of hypergeometric terms
(see \cite{AZ06} and \cite{Kbook}*{Chapters 6 and 7}).
For convenience we write
\begin{equation*}
S_q(k) =  \sum_{l_{1},l_{2}}H_q(l_{1},l_{2},k)= 2^{-4q} d_q(k).
\end{equation*}
First, we will handle the case where $k=q$.
\begin{lem}
\label{lem: k=q vanish}
For all $q \ge 2$ we have $S_q(q) = 0$.
\end{lem}
\begin{proof}
We have
\begin{align*}
d_q(q) = \sum_{l=0}^{q} a_q(l) a_q(q-l) b_q(l,q-l,0) &=
\frac{1}{(q!)^2} \sum_{l=0}^q {\binom{q}{l}}^2 \cdot \frac{(2q-2l)!(2l)!}{(2l-1)(2q-2l-1)} \\
&=\sum_{l=0}^q \binom{2q-2l}{q-l}  \frac{1}{2q-2l-1} \cdot \binom{2l}{l} \frac{1}{2l-1}.
\end{align*}
We write $ \phi(x) = \sum_{l=0}^{\infty} \binom{2l}{l} \frac{1}{2l-1} x^l = - \sqrt{1-4x}$. Then $\sum_{q=0}^{\infty} d_q(q) x^q = \phi(x)^2 = 1-4x$, showing that $d_q(q) = 0$ for all $q \ge 2$, whence the claim.
\end{proof}

Next, we prove a recurrence relation for $S_q(k)$.
\begin{lem}\label{lem:recursion_formula}
For all $q\ge1$ and all $k \ne q+2$ we have
\begin{equation*}
\frac{q^{2}}{8(2k-2q-3)(k-q-2)}S_{q}(k)+\frac{4kq-4q^{2}+2k-7q-4}{4(2k-2q-3)(k-q-2)}\cdot S_{q+1}(k)+S_{q+2}(k)=0.
\end{equation*}
\end{lem}
\begin{proof}
Let us begin by defining some rational functions in 4 variables. Let
\begin{align*}
Q_q^{(1)}(l_1,l_2,k) &
=4q^2 (l_1-l_2-k)  +4 q k^3
+  8qk^2(4-l_1-l_2)
\\ &\qquad
+4q k \left(l_1-l_2\right)^2 + 2q k(4l_2+14l_2 -11)
+2q (2l_1 l_2 +3l_1-9l_2+3)
\\&\qquad
+4k(2k+1)(3-2l_1)(2l_2-1)
+12l_1l_2-6l_1-18l_2+9,\\
Q_q^{(2)}(l_1,l_2,k) &=8l_1^{2}k-4l_1^{2}q+4l_1l_2q-12l_1kq+4l_1q^{2}-4l_2q^{2}+4kq^{2}-8l_1^{2}+4l_1l_2-12l_1k \\
&\qquad+10l_1q-6l_2q+10kq+6l_1-2l_2+4k-2q-1
\end{align*}
and
\begin{equation*}
Q_q(l_1,l_2,k)=32k(2k-2q-3)(k-q-2)\cdot(2q-l_1-l_2-k+1)_4.
\end{equation*}
Define also
\begin{equation*}
  R_q^{(1)}(l_1,l_2,k)=\frac{Q_q^{(1)}(l_1,l_2,k)(1/2+q-l_1)(l_1+l_2-k)(l_1-l_2+k)}{Q_q(l_1,l_2,k)}
\end{equation*}
and
\begin{equation*}
  R_q^{(2)}(l_1,l_2,k)=-\frac{Q_q^{(2)}(,l_1,l_2,k)(1/2+q-l_2)(l_1+l_2-k)(l_2-l_1+k)}{Q_q(l_1,l_2,k)}.
\end{equation*}

Applying Zeilberger's algorithm yields the following identity
of rational functions, which can be verified directly by expanding (and should be interpreted in the usual way at the poles):

\begin{align*}
  \frac{q^{2}}{8(2k-2q-3)(-q+k-2)}&+\frac{4kq-4q^{2}+2k-7q-4}{4(2k-2q-3)(-q+k-2)}\cdot\frac{H_{q+1}(l_1,l_2,k)}{H_q(l_1,l_2,k)}+\frac{H_{q+2}(l_1,l_2,k)}{H_q(l_1,l_2,k)}\\
  &=R_q^{(1)}(l_1+1,l_2,k)\cdot\frac{H_q(l_1+1,l_2,k)}{H_q(l_1,l_2,k)}-R_q^{(1)}(l_1,l_2,k) \\
  &\qquad+R_q^{(2)}(l_1,l_2+1,k)\cdot\frac{H_q(l_1,l_2+1,k)}{H_q(l_1,l_2,k)}-R_q^{(2)}(l_1,l_2,k).
\end{align*}
Therefore, after multiplying both sides by $H_q(l_1,l_2,k)$, one gets
\begin{align*}
\frac{q^{2} H_q(l_1,l_2,k)}{8(2k-2q-3)(-q+k-2)} & +\frac{4kq-4q^{2}+2k-7q-4}{4(2k-2q-3)(-q+k-2)}H_{q+1}(l_1,l_2,k)+H_{q+2}(l_1,l_2,k)\\
&=G_q^{(1)}(l_1+1,l_2,k)-G_q^{(1)}(l_1,l_2,k)+G_q^{(2)}(l_1,l_2+1,k)-G_q^{(2)}(l_1,l_2,k),
\end{align*}
where $G_q^{(1)}(l_1,l_2,k)=R_q^{(1)}(l_1,l_2,k)\cdot H_q(l_1,l_2,k)$, and $G_q^{(2)}(l_1,l_2,k)=R_q^{(2)}(l_1,l_2,k)\cdot H_q(l_1,l_2,k)$. Tedious but routine manipulations show that $G_q^{(1)}$ and $G_q^{(2)}$ are well-defined at the poles of $R_q^{(1)}$ and $R_q^{(2)}$.
We can now sum over all $l_1,l_2$ on both sides, noting that $H_q$ (and therefore $G_q^{(1)}$ and $G_q^{(2)}$) vanish for $|l_1|$ or $|l_2|$ sufficiently large, and get
$$
\frac{q^{2}}{8(2k-2q-3)(-q+k-2)}S_{q}(k)+\frac{4kq-4q^{2}+2k-7q-4}{4(2k-2q-3)(-q+k-2)}\cdot S_{q+1}(k)+S_{q+2}(k)=0,
$$
as claimed.
\end{proof}
Now Proposition~\ref{prop:For-all-,sum(H_k)} easily follows from Lemma~\ref{lem:recursion_formula}, by induction.
\begin{proof}[Proof of Proposition~\ref{prop:For-all-,sum(H_k)}~(a)]

We proceed by induction on $q$. For the base case note that
\begin{equation*}
P_{1}(x,y,z)=2(x+z)^{2}
\end{equation*}
whence, recalling \eqref{eq: Pq and Sq rel} and the relation $S_q(k) = 2^{-4q} d_q(k)$,
\begin{equation*}
\left(S_{1}(0)-\frac{4}{2^{4}}\right)+\left(\frac{4}{2^{4}}-S_{1}(1)\right)y^{2}+\sum_{k\ge2}S_{1}(k)y^{2k}=\frac{1}{2^{4}}P_{1}(-1,y,1)=0.
\end{equation*}
This implies that
\begin{equation*}
S_{1}(k)=\begin{cases}
0, & \text{for }k\ge2,\\
\frac{1}{4}, & \text{for }k=1,\\
\frac{1}{4}, & \text{for }k=0,
\end{cases}
\end{equation*}
which is exactly the case $q=1$.
Similarly, one verifies the formula for $q=2$.

Using now Lemma~\ref{lem:recursion_formula}, it is clear that we have $S_{q+2}(k)=0$ for all $2\le k<q+2$.
By Lemma~\ref{lem: k=q vanish}, this also holds for $k=q+2$. By definition, $S_{q+2}(k)=0$ for $k > q+2$. It remains to consider the cases $k=0,1$. Assume that
\begin{equation*}
\begin{array}{cc}
S_{q}(0)=2^{-4q} c_q, &S_{q}(1)= 2^{-4q}(2q-1)c_q  \\
S_{q+1}(0)=2^{-4(q+1)} c_{q+1}, &S_{q+1}(1)= 2^{-4(q+1)}(2q+1)c_{q+1} .
\end{array}
\end{equation*}
Then from Lemma~\ref{lem:recursion_formula} we have
\begin{align*}
-S_{q+2}(0)&=\frac{q^{2}}{8(2q+3)(q+2)}\cdot\frac{1}{2q\cdot \binom{2q}{q}}-\frac{4q^{2}+7q+4}{4(2q+3)(q+2)}\cdot\frac{1}{(2q+2)\cdot \binom{2q+2}{q+1}} \\
&=-\frac{1}{2(q+2)\cdot \binom{2q+4}{q+2}}
\end{align*}
and similarly
\begin{align*}
-S_{q+2}(1)&=\frac{q^{2}}{8(2q+1)(q+1)}\cdot\frac{2q-1}{2q\cdot \binom{2q}{q}}-\frac{4q^{2}+3q+2}{4(2q+1)(q+1)}\cdot\frac{2q+1}{(2q+2)\cdot \binom{2q+2}{q+1}}\\
&=-\frac{1}{4 \binom{2q+2}{q+1}}
\end{align*}
as claimed.
\end{proof}

\subsection{Proof of Proposition~\ref{prop:For-all-,sum(H_k)}~\ref{prop: sum H_k b}}

The development is very similar to that of the previous section, and we shall accordingly give less detail. We define
\begin{equation*}
S'_q(k) = \sum_{l_{1},l_{2}}H'_q(l_{1},l_{2},k)
\end{equation*}
and notice that $S'_q(k) = 2^{-4q} d'_q(k)$. We begin with a recurrence relation, similar to before.
\begin{lem}
For all $q\ge1$ and all $k \ne 2q-1$, we have
\begin{equation*}
\frac{q(k-2q-1)}{2(2k-2q-1)(k-2q+1)}S'_{q}(k)+S'_{q+1}(k)=0.
\end{equation*}
\end{lem}
\begin{proof}
This time we define
\begin{align*}
Q'_{1}(q,l_1,l_2,k)&=2l_2k-4kq+4l_2q-k+2l_2-2q-1, \\
Q'_{2}(q,l_1,l_2,k)&=2k^{2}-2l_2k-4l_2q+3k-2l_2+2q+1
\end{align*}
and
\begin{equation*}
Q'(q,l_1,l_2,k)=4k(2k-2q-1)(k-2q+1)(2q-l_1-l_2-k)(2q-l_1-l_2-k+1).
\end{equation*}
Define also
\begin{equation*}
  R'_{1}(q,l_1,l_2,k)=\frac{Q'_{1}(q,l_1,l_2,k)(1/2+q-l_1)(l_1+l_2-k)(l_1-l_2+k)}{Q^{'}(q,l_1,l_2,k)}
\end{equation*}
and
\begin{equation*}
  R'_{2}(q,l_1,l_2,k)=\frac{Q'_{2}(q,l_1,l_2,k)(1/2+q-l_2)(l_1+l_2-k)(l_2-l_1+k)}{Q^{'}(q,l_1,l_2,k)}.
\end{equation*}
Applying Zeilberger's algorithm again yields
\begin{align*}
\frac{q(k-2q-1)}{2(2k-2q-1)(k-2q+1)}H'_q(l_1,l_2,k)+H'_{q+1}(l_1,l_2,k) &= G'_{1}(q,l_1+1,l_2,k)-G'_{1}(q,l_1,l_2,k) \\
&+G'_{2}(q,l_1,l_2+1,k)-G'_{2}(q,l_1,l_2,k)
\end{align*}
where $G'_{1}(q,l_1,l_2,k)=R'_{1}(q,l_1,l_2,k)\cdot H'_q(l_1,l_2,k)$,
and $G'_{2}(q,l_1,l_2,k)=R'_{2}(q,l_1,l_2,k)\cdot H'_q(l_1,l_2,k)$.
We can now sum over all $l_1,l_2$ on both sides and get the result.
\end{proof}
Proposition~\ref{prop:For-all-,sum(H_k)}~(b) now follows by induction, as before.


\end{document}